\documentclass[11pt,english]{article}
\usepackage{mathtools}
\usepackage{mathrsfs}
\usepackage{graphicx}
\usepackage{graphics}
\usepackage{setspace}
	\usepackage{color}
\usepackage{amsmath,amssymb,amsthm}
\usepackage{epsfig}
\usepackage{babel}
\usepackage{tikz}
\usepackage{tikz-cd}
\usetikzlibrary{arrows,shapes,automata,backgrounds,petri,positioning,scopes,matrix, calc, plothandlers,plotmarks,patterns}
\usepackage{enumerate}
\usepackage{float}
\usepackage[utf8]{inputenc}
\usepackage[T1]{fontenc}
\usepackage[document]{ragged2e}
 \usepackage{verbatim}
\usepackage{tikz-3dplot}
\usepackage{pgfplots}
\usepackage{tkz-graph}
\GraphInit[vstyle = Normal]
\tikzset{
  LabelStyle/.style = {minimum width = 2em, 
                        text = red, font = \bfseries },
  VertexStyle/.append style = { inner sep=2pt,
                                font = \Large\bfseries, fill},
  EdgeStyle/.append style = {->, bend left} }

\newtheorem{thm}{Theorem}[section]
\numberwithin{equation}{section} 
\numberwithin{figure}{thm} 
\theoremstyle{plain}
\newtheorem*{thm*}{Theorem}
\theoremstyle{definition}
\theoremstyle{plain}
\newtheorem{thm_A}{Theorem}
\newtheorem*{defn*}{Definition}

\theoremstyle{plain}

\theoremstyle{plain} 

\theoremstyle{plain}
\newtheorem{prop}[thm]{Proposition} 
\theoremstyle{definition}
\newtheorem{ex}[thm]{Example}
\theoremstyle{remark}
\newtheorem{rem}[thm]{Remark}
\theoremstyle{plain}

\theoremstyle{plain}

\theoremstyle{plain}
\newtheorem{lem}[thm]{Lemma}
\newtheorem*{lem*}{Lemma} 
\theoremstyle{definition}
\newtheorem{defn}[thm]{Definition}
\newtheorem*{acknowledgment}{Acknowledgment}
\newtheorem*{acknowledgment*}{Addentum}
\newtheorem*{ques}{Question}
\theoremstyle{plain}
\newtheorem*{ex*}{Example}
\theoremstyle{plain}

\AtBeginDocument{
  
}
\begin{document}
\pgfdeclarelayer{background}
\pgfsetlayers{background,main}
\title{The Herzog-Sch\"onheim conjecture for finitely generated groups}
\author{Fabienne Chouraqui}

\date{}

\maketitle
\begin{abstract}
Let $G$ be a group and $H_1$,...,$H_s$ be subgroups of $G$ of  indices $d_1$,...,$d_s$ respectively. In 1974, M. Herzog and J. Sch\"onheim conjectured that if $\{H_i\alpha_i\}_{i=1}^{i=s}$,  $\alpha_i\in G$, is a coset partition of $G$, then $d_1$,..,$d_s$ cannot be distinct. We consider the  Herzog-Sch\"onheim conjecture for free groups of finite rank and develop a new combinatorial approach, using covering spaces. We define $Y$ the space of coset partitions of $F_n$  and show $Y$ is a metric space with interesting properties. We give some sufficient conditions on the coset partition that ensure the  conjecture is satisfied   and moreover has a neighborhood $U$ in $Y$ such that all the partitions in $U$ satisfy also the conjecture.
\end{abstract}
\maketitle
\section{Introduction}
Let $G$ be a group and $H_1$,...,$H_s$ be subgroups of $G$.  If there exist  $\alpha_i\in G$ such that $G= \bigcup\limits_{i=1}^{i=s}H_i\alpha_i$, and the sets  $H_i\alpha_i$, $1 \leq i \leq s$,  are pairwise disjoint, then  $\{H_i\alpha_i\}_{i=1}^{i=s}$ is \emph{a coset partition of $G$}  (or a \emph{disjoint cover of $G$}). In this case,    all the subgroups  $H_1$,...,$H_s$ can be assumed to be of  finite index in  $G$ \cite{newman,korec}. We denote by $d_1$,...,$d_s$ the indices of $H_1$,...,$H_s$ respectively. The coset partition $\{H_i\alpha_i\}_{i=1}^{i=s}$ has  \emph{multiplicity} if $d_i=d_j$ for some $i \neq j$. The  Herzog-Sch\"onheim conjecture is true for the group $G$, if any coset partition of $G$ has multiplicity.

\setlength\parindent{10pt}	If $G$ is the infinite cyclic group $\mathbb{Z}$, a coset partition of $\mathbb{Z}$ is 
$\{d_i\mathbb{Z} +r_i\}_{i=1}^{i=s}$, $r_i \in \mathbb{Z}$,  with  
each $d_i\mathbb{Z} +r_i$ the residue class of $r_i$ modulo $d_i$.
These coset partitions of $\mathbb{Z}$ were first introduced by P. Erd\H{o}s \cite{erdos1} and he conjectured that if $\{d_i\mathbb{Z} +r_i\}_{i=1}^{i=s}$, $r_i \in \mathbb{Z}$, is a  coset partition of $\mathbb{Z}$, then  the largest index $d_s$ appears at least twice.  Erd\H{o}s' conjecture was 
proved independently by H. Davenport, L. Mirsky, D. Newman and R.Rado using analysis of complex function \cite{erdos2,newman,znam}. Furthermore, it was proved that  the largest index $d_s$ appears at least $p$ times,  where $p$ is the smallest prime dividing $d_s$ \cite{newman,znam,sun2}, that each index $d_i$  divides another index $d_j$, $j\neq i$, and  that each index $d_k$ that does not properly divide any other index  appears at least twice \cite{znam}. We refer also to \cite{por1,por2,por3,por4,sun3} for more details on  coset partitions of $\mathbb{Z}$ (also called covers of $\mathbb{Z}$ by arithmetic progressions) and to \cite{ginosar} for a proof of the Erd\H{o}s' conjecture using group representations. 

In 1974, M. Herzog and J. Sch\"onheim extended Erd\H{o}s' conjecture for arbitrary groups and conjectured that if $\{H_i\alpha_i\}_{i=1}^{i=s}$,  $\alpha_i\in G$, is a  coset partition of $G$, then $d_1$,..,$d_s$ cannot be distinct. In the 1980's, in a series of papers,  M.A. Berger, A. Felzenbaum and A.S. Fraenkel studied  the Herzog-Sch\"onheim conjecture \cite{berger1, berger2,berger3} and in \cite{berger4} they proved the conjecture is true for the pyramidal groups, a subclass of the finite solvable groups. Coset partitions of finite groups with additional assumptions on the subgroups of the partition have been extensively studied. We refer to \cite{brodie,tomkinson1, tomkinson2,sun}. In \cite{schnabel}, the authors very recently proved that the conjecture is true for all groups of order less than $1440$. 

	The common approach to the Herzog-Sch\"onheim conjecture is to study it in finite groups. Indeed, given any group $G$, every coset partition of $G$ induces a coset partition of a particular finite  quotient group of $G$ with the same indices (the quotient of $G$  by the intersection of the normal cores of the subgroups from the partition)  \cite{korec}. In this paper, we adopt a completely  different approach to the   Herzog-Sch\"onheim conjecture. Instead of  finite groups,  we consider  free groups of finite rank and develop a new approach to the problem.  This approach has two advantages. First, it permits to  study the   Herzog-Sch\"onheim conjecture in the  free groups of finite rank for their own sake and secondly it provides some answers for every finitely generated  group.  Indeed, we show any coset partition of a finitely generated  group  $G$ induces a  coset partition of $F$ with the same indices. \\

In order  to  study the   Herzog-Sch\"onheim conjecture in  free groups of finite rank, we  use the machinery of covering spaces. The  fundamental group of the  bouquet with $n$ leaves (or the wedge sum of $n$ circles),  $X$,  is $F_n$, the  free group of finite rank $n$. As  $X$ is a ``good'' space (connected, locally path connected and semilocally $1$-connected), $X$ has a  universal covering  which can be identified with the Cayley graph of $F_n$,  an infinite simplicial tree. Furthermore, there exists a  one-to-one correspondence between the subgroups of $F_n$ and the covering spaces (together with a chosen point) of $X$.  Using  these  covering spaces, J. Stallings  gave  a topological proof of some classical results about finitely generated subgroups of $F_n$ and  introduced the notion of folding \cite{stallings}.  \\

 For any  subgroup $H$ of $F_n$ of finite index $d$, there exists  a $d$-sheeted covering space  $(\tilde{X}_H,p)$  with a fixed basepoint, which is also a combinatorial object.  Indeed, the underlying graph of $\tilde{X}_H$  is a directed labelled graph, with $d$ vertices,  that can be seen as  a finite complete bi-deterministic automaton; fixing the start and the end state at the basepoint, it recognises the set of elements in $H$.  It is   called \emph{the Schreier coset diagram  for $F_n$ relative to the subgroup  $H$} \cite[p.107]{stilwell} or  \emph{the Schreier automaton for $F_n$ relative to the subgroup $H$} \cite[p.102]{sims}. The $d$ vertices (or states) correspond to the $d$ right cosets of $H$,  each edge (or transition) $Hg \xrightarrow{a}Hga$, $g \in F_n$, $a$ a generator of $F_n$,  describes the right action of $a$ on  $Hg$.   \emph{The transition group $T$} of the Schreier automaton for $F_n$ relative to  $H$  describes the action of $F_n$ on the set of the $d$ right cosets of $H$, and is generated by $n$ permutations. The group    $T$ is a subgroup of $S_d$ such that   $T \simeq\,^{F_n}\big/_{N_H}$, where $N_H= \bigcap\limits^{}_{g \in F_n}g^{-1}Hg$ is the normal core of $H$.     \\

 Given $H \leq F_n$ of index $d$,  we consider the  covering $(\tilde{X}_H,p)$ of $X$ with basepoint $\tilde{x}_0$, from several points of view in parallel: as a covering, as a  Schreier coset diagram  for $F_n$ relative to the subgroup  $H$, as a  complete automaton. We use  the following terminology. Let $\tilde{X}_H$ denote  the Schreier coset diagram  for $F_n$ relative to the subgroup  $H$. Let $\tilde{x}_0,..., \tilde{x}_{d-1}$ be the $d$ vertices in  $\tilde{X}_H$. Let   $t_i$  denote the label of a minimal path from $\tilde{x}_0$ to any vertex $\tilde{x}_i$, $1\leq i \leq d-1$. Let $\mathscr{T}=\{1, t_i\mid 1 \leq i\leq d-1\}$.  As $\tilde{X}_H$ is the Schreier coset diagram  for $F_n$ relative to the subgroup  $H$,  $\tilde{x}_0$ represents the subgroup $H$ and $\tilde{x}_1,...,\tilde{x}_{d-1}$ represent the cosets  $Ht_i$ accordingly.  We call $\tilde{X}_{H}$  \emph{the  Schreier graph  of $H$},  with this correspondence between the vertices  $\tilde{x}_0, \tilde{x}_1,...,\tilde{x}_{d-1}$ and the cosets  $H$, $Ht_1$,...,$Ht_s$ accordingly. \\

 The intuitive idea behind our approach  is  as follows. Let  $\{H_i\alpha_i\}_{i=1}^{i=s}$ be a coset  partition of $F_n$, $n \geq 2$,  with $H_i<F_n$ of index $d_i>1$, $\alpha_i \in F_n$, $1 \leq i \leq s$. Let $w \in F_n$.   For every Schreier  graph $\tilde{X}_{i}$, $1 \leq i \leq s$,  we define a  graph  $\tilde{X}_{i}(w)$ that describes  the action of $w$ on the vertices of $\tilde{X}_{i}$. We denote by $o_{*i}(w)$ the minimal natural number, $1 \leq o_{*i}(w) \leq d_i$, such that $w^{o_{*i}(w)}$  is  a loop  at the vertex $H_i\alpha_i$. Combining in some way the $s$  graphs, $\tilde{X}_{1}(w)$,...,$\tilde{X}_{s}(w)$, we define a new finite  graph, \emph{the HS-colored  graph}, to describe graphically the coset partition of $F_n$. \\

We show the HS-colored  graph is the disjoint union of  loops of the same length, 
 that satisfies many interesting properties. In particular, we show that each   loop $\ell$ describes  a coset  partition of $\mathbb{Z}$, with subgroups of indices $\{o_{*i}(w) \mid i \in I_{\ell}\}$, where $I_{\ell} \subseteq\{1,2,...,s\}$ and we can then apply  the results on the Erd\H{o}s' conjecture. Indeed, if 
$o_{max}(w)=max\{o_{*i}(w) \mid 1 \leq  i \leq s\}$, then  $o_{*i}(w)=o_{max}(w)$,  for at least $p$ values of $i$, where 
$p$ is the smallest prime dividing $o_{max}(w)$. Since each $o_{*i}(w)$ is the length of a cycle in some permutation in $T_{i}$, we can formulate our results in terms of cycles in the permutation groups $T_1$,...,$T_s$.  In particular, if  $o_{max}=max\{o_{max}(w) \mid w \in F_n\}=d_s$, we have the following first result. 




 \begin{thm_A}\label{theo0}
 Let $F_n$ be the free group on $n \geq 2$ generators. Let $\{H_i\alpha_i\}_{i=1}^{i=s}$ be a coset  partition of $F_n$ with $H_i<F_n$ of index $d_i$, $\alpha_i \in F_n$, $1 \leq i \leq s$, and $1<d_1 \leq ...\leq d_s$.  Let $\tilde{X}_{i}$ denote the  Schreier  graph of $H_i$, with  transition group  $T_i$, $1 \leq i\leq s$.
     If there exists a $d_s$-cycle in $T_s$, then  the index $d_s$ appears in the partition at least  $p$ times, where $p$ is the smallest prime dividing $d_s$.  Furthermore, all the subgroups $H_i$ with same index  are isomorphic.
  \end{thm_A} 
The transition group    $T_s$ is a subgroup of the symmetric group $S_{d_s}$, generated by $n \geq 2$ permutations. Dixon proved that the probability that a random pair of  elements of $S_n$ generate $S_n$ approaches $3/4$ as $n \rightarrow\infty$, and the probability that they generate $A_n$ approaches $1/4$ \cite{dixon}. As  $d_s\rightarrow\infty$, the probability that $T_s$ is the symmetric group $S_{d_s}$ approaches $3/4$. So, asymptotically,  the probability that there exists a $d_s$-cycle in $T_s$ is greater than $3/4$. If  $T_s$ is cyclic,  there exists a $d_s$-cycle in $T_s$, since $d_s$ divides the order of $T_s$. That is, Theorem \ref{theo0} is satisfied with very high probability and the conjecture is ``asymptotically satisfied with probability greater than $3/4$''  for free groups of finite rank. \\

 Theorem \ref{theo1} provides a list of conditions on a coset partition  that ensure multiplicity. Let  $k=max\{o_{max}(w) \mid w \in F_n\}$,  $k$ is the  maximal length of a cycle in $\bigcup\limits_{i=1}^{i=s}T_i$. Let $p$ denote the smallest prime dividing $k$.  We show there exists  $u \in F_n$ such that  $o_{max}(u)=k$ and $\# \geq 2$, where   $\#=\mid \{1 \leq i\leq s \mid o_{*i}(u)=k\}\mid$. Using this notation, we have the following result. 
   \begin{thm_A}\label{theo1}
  Let $F_n$ be the free group on $n \geq 2$ generators. Let $\{H_i\alpha_i\}_{i=1}^{i=s}$ be a coset  partition of $F_n$ with $H_i<F_n$ of index $d_i$, $\alpha_i \in F_n$, $1 \leq i \leq s$, and $1<d_1 \leq ...\leq d_s$.  Let $r$ be an integer, $4 \leq r \leq s-1$.
     If $k$, $p$ and $\#$, as defined above, satisfy one of the following conditions:
  \begin{enumerate}[(i)]
  \item  $k>d_{s-2}$.
 \item    $k>d_{s-3}$,  $p\geq 3$.
   \item    $k>d_{s-3}$,  $p=2$, and $\#=2$ or $\# \geq 4$.
   \item  $k>d_{s-r}$ and   $p\geq r$, or  $\#=p$, or   $\# \geq r+1$.  
    \end{enumerate} 
  Then the coset partition $\{H_i\alpha_i\}_{i=1}^{i=s}$ has multiplicity. 
  \end{thm_A}
Assuming that  $k > d_1$,  there is a finite number of  cases that are not covered by Theorem \ref{theo1}. For
    $k>d_{s-3}$, there is only one  case which is not covered: $p=2$,  $\#=3$.  
   For $k>d_{s-4}$, there are three cases which are not covered: $p=2$, $\#=3,4$ and
   $p=3$, $\#=4$. As $r$ grows, the number of cases not covered by Theorem \ref{theo1} grows also, but as $s$ is finite, the total number of cases not covered by Theorem \ref{theo1} is finite.\\

  In the following Theorem, we give a  condition that ensures  the same  subgroup appears at least twice in the partition. 
   \begin{thm_A}\label{theo2}
 Let $F_n$ be the free group on $n \geq 2$ generators. Let $\{H_i\alpha_i\}_{i=1}^{i=s}$ be a coset  partition of $F_n$ with $H_i<F_n$ of index $d_i>1$, $\alpha_i \in F_n$, $1 \leq i \leq s$.\\  \noindent  If one of the following conditions holds:
  \begin{enumerate}[(i)]
  \item There exist  $1 \leq j,k \leq s$ such that $\bigcap \limits_{i=1}^{i=s}\alpha_i^{-1}H_i\alpha_i \subsetneqq \bigcap \limits_{i\neq j,k}\alpha_i^{-1}H_i\alpha_i$.
   \item There exist  $1 \leq j,k \leq s$ such that $\ell cm(d_j,d_k)$ does not divide the index $[F_n: \bigcap \limits_{i\neq j,k}\alpha_i^{-1}H_i\alpha_i]$.
      \end{enumerate} 
  Then $\{H_i\alpha_i\}_{i=1}^{i=s}$ has multiplicity. Furthermore,  $H_j=H_k$.
  \end{thm_A}
 
 Inspired by \cite{sikora}, in which the author defines the space of left orders of a left-orderable group and show it is a compact and totally disconnected metric space,  we define $Y$ to be  the space of coset partitions of $F_n$ (under some equivalence relation) and  show $Y$ is a metric space. In our case, the metric defined  induces the discrete topology. 
 \begin{thm_A}\label{theo3}
   Let $F_n$ be the free group on $n \geq 2$ generators. Let $Y$  be  the space of coset partitions of $F_n$ (under some equivalence relation). 
   Then $Y$ is a metric space with a metric $\rho$ and $Y$ is   (topologically) discrete 
   \end{thm_A}
We  show that for  each coset partition of $F_n$, which satisfies one of the conditions in Theorems \ref{theo0} or \ref{theo1},  there exists a  neighborhood $U$ in $Y$ such that all the coset partitions in $U$ have multiplicity.

\begin{thm_A}\label{theo4}
   Let $F_n$ be the free group on $n \geq 2$ generators. Let $Y$  be  the space of coset partitions of $F_n$ (under some equivalence relation) with metric $\rho$. Let $P_0=\{H_i\alpha_i\}_{i=1}^{i=s}$ be in $Y$, with  $1<d_1 \leq ...\leq d_s$. 
\begin{enumerate}[(i)]
\item If  $P_0$ satisfies the condition of Theorem \ref{theo0}, then every $P \in Y$ with $\rho(P, P_0)< \frac{1}{2}$ satisfies  the same condition and hence has multiplicity.
\item If  $P_0$ satisfies $(i)$ or $(ii)$ of Theorem \ref{theo1}, with some $2 \leq r \leq s-1$, then every $P \in Y$ with $\rho(P, P_0)< 2^{-(r+1)}$ satisfies  the same condition and hence has multiplicity.
\item If  $P_0$ satisfies $(iii)$ or $(iv)$ of Theorem \ref{theo1}, with some $2 \leq r \leq s-1$, then every $P \in Y$ with $\rho(P, P_0)< 2^{-(r+1)}$ has multiplicity.
\end{enumerate} 
 \end{thm_A}

Theoerem \ref{theo4} implies that given  a ``small''  subgroup $H_s$ (or some ``small'' subgroups) satisfying one of the conditions cited above,  then for  every  $\alpha_s \in F_n$, any completion of   $H_s\alpha_s$,  to a coset partition of $F_n$ with $H_s$ being the ``smallest'' (greatest index) subgroup has multiplicity.\\

We now turn to finitely generated groups in general.
 \begin{thm_A}\label{theo5}
  Let $G \simeq\, ^{F_n}\big/_{K}$,  with canonical epimorphism $\pi: F_n \rightarrow G$.  Let $\{K_ig_i\}_{i=1}^{i=s}$ be a coset  partition of $G$ with $K_i<G$ of index $d_i>1$, $g_i \in G$, $1 \leq i \leq s$. Let $H_i=\pi^{-1}(K_i)$ and $\alpha_i=\pi^{-1}(g_i)$.
  Then there exists  $\{H_i\alpha_i\}_{i=1}^{i=s}$ an induced coset  partition of $F_n$,  with $H_i<F_n$ of index $d_i$, $\alpha_i \in F_n$, $1 \leq i \leq s$.  Furthermore, if for  $\{H_i\alpha_i\}_{i=1}^{i=s}$, one of the conditions from  Theorem \ref{theo0}, \ref{theo1}, or  \ref{theo2} is satisfied, then the coset partition $\{K_ig_i\}_{i=1}^{i=s}$ has multiplicity.
  \end{thm_A}
    The paper is organized as follows.  The two first sections are background sections: Section $1$ on free groups and covering spaces and  Section $2$ on graphs and automatons. In Section $3$, we describe the equivalent approaches to covering spaces in the special case of free groups and in particular we present the Schreier graph and its properties. In Section $4$, we give a graphical description of  the coset partition of $F_n$ in terms of the Schreier graphs, and in particular 
       in the case of $\mathbb{Z}$. Sections $6$ is devoted to the proofs of the main results. In Section $7$, 
 we introduce the space of coset partitions of $F_n$, an  action of $F_n$ on it, a metric  and prove Theorems \ref{theo2}, \ref{theo3} and \ref{theo4}. In Section 8, we consider the case of any finitely generated group and prove Theorem \ref{theo5}.    In many places, we write the HS conjecture instead of  the Herzog-Sch\"onheim conjecture.

\begin{acknowledgment}
 I am very grateful to Yuval Ginosar for introducing me to the Herzog-Sch\"onheim conjecture,   to Luis Paris for suggesting me  to study the conjecture in free groups, and to Aryeh Juhasz for his useful comments on the paper.
\end{acknowledgment}


\section{Free Groups and Covering spaces}\label{sec_FREE}

\subsection{Free Groups}
For more details, we refer the reader to the books \cite[Chapter 1]{lyndon}, \cite[Section 1.4]{magnus}, \cite[Chapter 1]{baumslag}.
Let $X$ be a non-empty  set. The group generated by $X$ is
 $\langle X\rangle =\{ x_{1}^{\epsilon_1} x_{2}^{\epsilon_2}..x_{n}^{\epsilon_n}\mid x_{i} \in X, \epsilon_i=\pm 1,n\geq0 \}$, this  is  the set of all words on $X$.
 If $X$ is finite, the group generated by $X$ is \emph{finitely generated}. A word $w=x_{1}^{\epsilon_1} x_{2}^{\epsilon_2}..x_{n}^{\epsilon_n}$ is \emph{reduced}  
if $x_{i}=x_{i+1}$ implies $\epsilon_i + \epsilon_{i+1}\neq 0$ and in this case  \emph{the length of $w$}, $\ell(w)$, is equal to $n$.  A word $w=x_{1}^{\epsilon_1} x_{2}^{\epsilon_2}..x_{n}^{\epsilon_n}$ is \emph{cyclically reduced}  
if $x_{1}=x_{n}$ implies $\epsilon_1 + \epsilon_{n}\neq 0$. If $F=\langle X \rangle$ and every non-empty reduced word $w \neq 1$ in $F$, then we say that  \emph{$X$ is a free set of generators of $F$} or that \emph{$F$ is free with base  $X$}. It holds that all bases for a given free group have the same cardinal and this common cardinal is termed \emph{the rank of the free group $F$} and denoted by $rank(F)$. The free group of rank $n$ is denoted by $F_n$ and $F_n \simeq F_m$ if and only if $n=m$. \\

The Nielsen-Schreier theorem states that every  subgroup of a free group is free. If $H$ is a  subgroup of $F_n$ of finite index $d$, then
$H$ is a free group of rank $d(n-1)+1$ and $H$ has non-trivial intersection with every non-trivial subgroup of $F_n$.
If $K$ is a  subgroup of $F_n$ of same  index $d$, then $K \simeq H$. Note that a finitely generated subgroup of $F_n$ is not necessarily of  finite index and that there exist infinitely generated subgroups in $F_n$. A set $S$  of representatives of the right cosets of a subgroup $H$ in a group $G$ generated by $X$ is called a \emph{(right) Schreier transversal} if whenever $x_{1}^{\epsilon_1} x_{2}^{\epsilon_2}..x_{n}^{\epsilon_n} \in S$, $x_i \in X$, then $x_{1}^{\epsilon_1} x_{2}^{\epsilon_2}..x_{n-1}^{\epsilon_{n-1}}$ is also in $S$. If $H$ is a subgroup of the free group $F_n$, then there exists a right Schreier transversal of $H$ in $F_n$.
If $G$ is a  group generated by a set of $n$ of its elements, then $G\simeq\,^{F_n}\big/_{N}$,  with $N \lhd F_n$.
If  $G\simeq\,^{F_n}\big/_{N}$, then the group $G$ has a presentation $\langle X \mid R \rangle$, where $X$ is a base of $F_n$ and the normal  closure of the set $R$ is  equal to $N$. If $X$ and $R$ are finite, the group $G$ is \emph{finitely presented}. 
\subsection{Covering spaces}
For more details, we refer the reader to the books \cite[Chapter 10]{rotman}, \cite[Chapter 1]{hatcher}, \cite[Chapter 3]{lyndon}, \cite{stilwell},  \cite{geogh}.
If $X$ is a topological space, then $(\tilde{X},p)$ 
is a  \emph{covering space} of $X$ if $\tilde{X}$ is a path connected topological space, $p: \tilde{X}\rightarrow X$ is an open continuous surjection and each point $x \in X$ has an open neighborhood $U_x$ such that $p^{-1}(U_x)$ is a disjoint union of  open sets in $\tilde{X}$, each of which is mapped homeomorphically onto $U_x$ by $p$. If $(\tilde{X},p)$ 
is a covering space of $X$, then $X$ is path connected. For each $x \in X$, the non-empty set $Y_x=p^{-1}(x)$ is called \emph{the fiber over $x$} and for all $x,x' \in X$, $\mid Y_x \mid =\mid Y_{x'}\mid $. The cardinal of a fiber is called \emph{the multiplicity} of the covering. One also says $(\tilde{X},p)$  is a \emph{$m$-sheeted covering} (\emph{$m$-fold cover}) of $X$.\\

 If $X$ and $\tilde{X}$ are spaces and there is a continuous map $p:\tilde{X} \rightarrow X$, with $p(\tilde{x}_0)=x_0$,  then there is an  induced homomorphism $p_*:\pi_1(\tilde{X},\tilde{x}_0) \rightarrow \pi_1(X,x_0)$. If $[f]$ is a loop class in $\tilde{X}$, then  $p_*([f])=[p(f)]$ \cite[Chapter 3]{rotman}.

\begin{thm}\cite[p.279-281]{rotman}\label{thm_cov_inject}
Let  $(\tilde{X},p)$ 
be a covering space of $X$. Let $G=\pi_1(X,x_0)$. Let $Y_0$ denote the fiber over $x_0$ and let  $\tilde{x}_{0} \in Y_0$. Then
\begin{enumerate}[(i)]
\item The induced homomorphism $p_*:\pi_1(\tilde{X},\tilde{x}_0) \rightarrow G$ is an injection.
\item $G$  acts transitively on $Y_0$.
\item  The stabilizer of $\tilde{x}_0$ is  $p_*(\pi_1(\tilde{X},\tilde{x}_{0}))$.
\item $\mid Y_0 \mid =[ G:p_*(\pi_1(\tilde{X},\tilde{x}_{0}))]$.
\item $p_*(\pi_1(\tilde{X},\tilde{x}_0))$ and  $p_*(\pi_1(\tilde{X},\tilde{x}_1))$ are conjugates subgroups of $G$,  for every $\tilde{x}_1 \in Y_0$.
\item If $K \leq G$ is conjugate to $p_*(\pi_1(\tilde{X},\tilde{x}_0))$, then there exists $\tilde{x}_1\in Y_0$ such that $K=p_*(\pi_1(\tilde{X},\tilde{x}_1))$.
\end{enumerate} 
\end{thm}  
 A  covering space  $(\tilde{X},p)$ of $X$ is \emph{regular} if $p_*(\pi_1(\tilde{X},\tilde{x}_{0})) \lhd G$. A \emph{Deck transformation}  of $\tilde{X}$ is a homeomorphism $h:\tilde{X} \rightarrow \tilde{X}$ with $p \circ h=p$. The group of Deck transformations, $\operatorname{Cov}(\tilde{X}/X)$, is called 
      \emph{The Deck group} and it  acts freely on $\tilde{X}$.
  If $(\tilde{X},p)$ is universal, that is $\tilde{X}$ is simply-connected, then   $\operatorname{Cov}(\tilde{X}/X) \simeq G$.  If $X$ is  locally path connected, $H=p_*(\pi_1(\tilde{X},\tilde{x}_0))$, then $\operatorname{Cov}(\tilde{X}/X)\simeq\, ^{N(H)}\big/_{H}$,  where $N(H)$ is the normaliser of $H$ in  $G$. So, $(\tilde{X},p)$ is regular if and only if  $\operatorname{Cov}(\tilde{X}/X) \simeq \,^{G}\big/_{H}$  \cite[p.289-294]{rotman}.\\

Two covering spaces $(\tilde{X},p)$ and $(\tilde{Y},q)$ of a space $X$ are \emph{equivalent} if there exists a homeomorphism $\varphi:\tilde{Y}\rightarrow\tilde{X}$  such that $q=p\varphi$. If $X$ is  a  locally path connected space, then the covering spaces $(\tilde{X},p)$ and $(\tilde{Y},q)$ of  $X$ are equivalent if and only if $p_*(\pi_1(\tilde{X},\tilde{x}_0))$ and $q_*(\pi_1(\tilde{Y},\tilde{y}_0))$ are conjugate subgroups of  $\pi_1(X,x_0)$,  with  $\tilde{x}_0\in p^{-1}(x_0)$ and $\tilde{y}_0\in q^{-1}(x_0)$.  If $X$ satisfies several conditions, a converse to Theorem \ref{thm_cov_inject} is true, that is given $x_0 \in X$ and a subgroup $H$ of  $\pi_1(X,x_0)$, there exists a  covering space  $(\tilde{X},p)$ of $X$, with  $p_*(\pi_1(\tilde{X},\tilde{x_0}))= H$ for a suitably chosen basepoint $\tilde{x}_0 \in \tilde{X}$. A space $X$ is \emph{semilocally 1-connected} if each $x\in X$ has an open neighborhood $U$ so that $i_*:\pi_1(U,x)\rightarrow \pi_1(X,x)$ is the trivial map (where $i:U\xhookrightarrow{} X$ is the inclusion). A connected and locally path connected space $X$ has a universal covering if and only if $X$ is semilocally $1$-connected.
\begin{thm}\cite[p.295-300]{rotman}\label{thm_cov_corresp}
Let  $X$ be a  connected, locally path connected and semilocally $1$-connected space. Let  $x_0 \in X$.  Let $H$ be a subgroup of $\pi_1(X,x_0)$. Then there exists  a covering space $(\tilde{X},p)$ of $X$,    and $\tilde{x}_{0}\in p^{-1}(x_0)$ such that  $p_*(\pi_1(\tilde{X},\tilde{x}_{0}))= H$.
 \end{thm}
\flushleft
The  free group of finite rank $n$, $F_n$,  is the fundamental group of  the bouquet with $n$ leaves, which is a space that satisfies all  the conditions of  Theorem \ref{thm_cov_corresp}. So,  we can take for granted the one-to-one correspondence between the subgroups of $F_n$ and the covering spaces (together with a chosen point) of the bouquet with $n$ leaves. 
\begin{ex} \label{ex_cov_subgrps}
Let $K= \langle b,a^2,ab^2a,aba^2ba,(ab)^3a \rangle$ be a subgroup of index $4$ in $F_2$. So, there exist a  covering space  $(\tilde{X}_K,p')$ over $X$ and a  chosen basepoint $\tilde{x}_{0_0}$ (or $\tilde{x}_{0_1}$) such that $p'_*(\pi_1(\tilde{X}_K,\tilde{x}_{0_0}))= K$. Note  that this covering is not regular, so $K$ is not normal in $F_2$.
  \begin{figure}[H] 
   \centering \scalebox{0.9}[0.8]{\begin{tikzpicture}
   \SetGraphUnit{4}
    \tikzset{VertexStyle/.append  style={fill}}
     \Vertex[L=$\color{blue}{\tilde{x}_{0_0}}$, x=-3,y=0]{A}
     \Vertex[L=$\tilde{x}_{1_0}$, x=0, y=0]{B}
   
   \Vertex[L=$\tilde{x}_{1_1}$, x=3, y=0]{C}
    \Vertex[L=$\color{blue}{\tilde{x}_{0_1}}$, x=6, y=0]{D}
   \Edge[label = a, labelstyle = above](A)(B)
    \Edge[label =a, labelstyle = below](B)(A)
   \Edge[label = b, labelstyle = above](B)(C)
   \Edge[label = b, labelstyle = below](C)(B)
    \Edge[label = b, labelstyle = above](B)(C)
     \Edge[label = a, labelstyle = above](C)(D)
   \Edge[label = a, labelstyle = below](D)(C)
   \Loop[dist = 2cm, dir = NO, label = b, labelstyle = left](A.west)
   \Loop[dist = 2cm, dir = SO, label = b, labelstyle = right](D.east)
 \end{tikzpicture}}\\
   $\downarrow^{p'}$\\
  \scalebox{0.9}[0.7]{ \begin{tikzpicture}
   \SetVertexMath
     \SetGraphUnit{4}
     \tikzset{VertexStyle/.append  style={fill}}
       \Vertex{x_0}
      \Loop[dist = 2cm, dir = NO, label = b , labelstyle = left ](x_0.west)
     \Loop[dist = 2cm, dir = SO, label = a , labelstyle = right  ](x_0.east)
    \end{tikzpicture}}
    \caption{The $4$-sheeted covering $(\tilde{X}_K,p')$ over $X$.}\label{fig-4-shheted-covering-K}
     \end{figure}
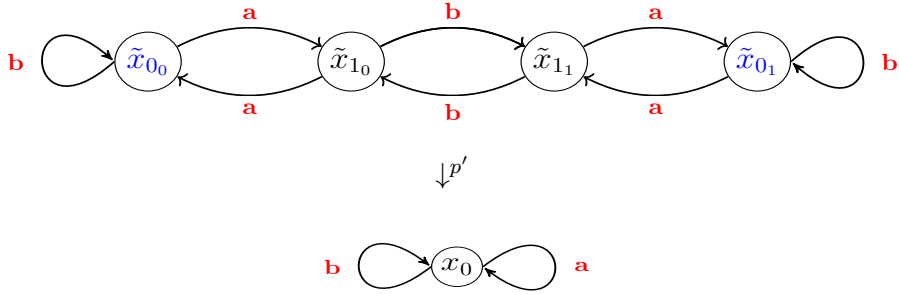\end{ex}
       \flushleft

\section{Graphs and Automata}\label{sec_graphs}
\subsection{Graphs}\label{subsec_graphs}
For more details, we refer the reader to \cite{serre}, \cite{stallings}, \cite{gross}.
A \emph{graph $\Gamma$} consits of two sets $V$ and $E$, and two maps $E \rightarrow V \times V $ $:$ $e \mapsto (i(e), t(e))$ and $E \rightarrow E$ $:$ $e \mapsto \bar{e}$, which satisfy the following conditions: $\forall e \in E$, $\bar{e}\neq e$, $\bar{\bar{e}}= e$ and $i(\bar{e})=t(e)$. An element $v \in V$ is called a \emph{vertex of $\Gamma$}, $e \in E$  is called an \emph{edge of $\Gamma$} and $\bar{e}$  is called the \emph{reverse  of $e$}. The vertex $i(e)$ is the \emph{initial vertex of $e$ (or origin)} and  $t(e)$ the \emph{terminal vertex of $e$ (or terminus)}. These two vertices are called the \emph{extremities of $e$} and two vertices are \emph{adjacent} if they are the extremities of some edge. An \emph{orientation} of $\Gamma$ consists of a choice of exactly one edge in each pair $\{e, \bar{e}\}$, and in this case $\Gamma$ is called a \emph{directed graph}. The  \emph{indegree$(v)$}  is the cardinal of $\{e \in E \mid t(e)=v\}$ and  the \emph{outdegree}$(v)$ is the cardinal of $\{e \in E \mid i(e)=v\}$.  A \emph{path $p$ in $\Gamma$} of length $n \geq 0$, with origin $v_1$ and terminus $v_n$ is an $n-$tuple of edges of $\Gamma$, $p=e_1e_2...e_n$, such that for $j=1,..,n-1$, $t(e_j)=i(e_{j+1})$ and $v_1=i(e_1)$, $v_n=t(e_n)$. A path $p$ is a \emph{loop at $v$} if $v$ is both its origin and its terminus. A loop of length $1$ at $v$ contributes $1$ to the indegree of $v$ and $1$ to its outdegree. 
\subsection{Automata}\label{subsec_automat}
We refer the reader to \cite[p.96]{sims}, \cite[p.7]{epstein}, \cite{pin,pin2}.
A \emph{finite state automaton} is a quintuple $(S,A,\mu,Y,s_0)$, where $S$ is a finite set, called the \emph{state set}, $A$ is a finite set, called the \emph{alphabet}, $\mu:S\times A \rightarrow S$ is a function, called the \emph{transition function}, $Y$ is a (possibly empty) subset of $S$ called the \emph{accept (or final) states}, and $s_0$ is called the \emph{start state}.  It is a directed  graph with vertices the states and each transition $s \xrightarrow{a} s'$ between states $s$ and $s'$ is an edge with label $a \in A$. The \emph{label of a path $p$} of length $n$  is the product $a_1a_2..a_n$ of the labels of the edges of $p$.
The  finite state automaton $M=(S,A,\mu,Y,s_0)$ is \emph{deterministic} if there is only one initial state and each state is the source of exactly one arrow with any given label from  $A$. In a deterministic automaton, a path is determined by its starting point and its label \cite[p.105]{sims}. It is \emph{co-deterministic} if there is only one final state and each state is the target of exactly one arrow with any given label from  $A$. The  automaton $M=(S,A,\mu,Y,s_0)$ is \emph{bi-deterministic} if it is both deterministic and co-deterministic. An automaton $M$ is \emph{complete} if for each state $s\in S$ and for each $a \in A$, there is exactly one edge from $s$ labelled $a$.
\begin{defn}
Let $M=(S,A,\mu,Y,s_0)$ be  a finite state automaton. Let $A^*$ be the free monoid generated by $A$. Let  $\operatorname{Map}(S,S)$ be  the monoid consisting of all maps from $S$ to $S$. The map $\phi: A \rightarrow \operatorname{Map}(S,S) $ given by $\mu$ can be extended in a unique way to a monoid homomorphism $\phi: A^* \rightarrow \operatorname{Map}(S,S)$. The range of this map is a monoid called \emph{the transition monoid of $M$}, which is generated by $\{\phi(a)\mid a\in A\}$. An element $w \in A^*$ is \emph{accepted} by $M$ if the corresponding element of $\operatorname{Map}(S,S)$, $\phi(w)$,  takes $s_0$ to an element of the accept states set $Y$. The set $ L\subseteq A^*$  recognized by $M$ is called \emph{the language accepted by $M$}, denoted by $L(M)$.
\end{defn}
In order to extend this definition to languages in  groups, one takes the alphabet $A$ of monoid generators to be closed under inverses and  each transition  $ s \xrightarrow{a} s'$ labelled by $a \in A$ induces a "reverse" transition $ s' \xrightarrow{x} s$, with $x=a^{-1}$.
\begin{ex}\label{ex_automaton_index3}
 Let  $A=\{a,b,x=a^{-1},y=b^{-1}\}$ be an alphabet. So, $A^*$ is $F_2$. The transition monoid  $T$ of the following complete automaton $M$  is generated by $\{\phi_a=\phi_x=(s_0,s_1),\phi_b=\phi_y=(s_1,s_2) \}$, that is  $S_3$.
\begin{figure}[H]
\centering  \scalebox{0.9}[0.8]{ \begin{tikzpicture}
   \SetGraphUnit{4}
    \tikzset{VertexStyle/.append  style={fill}}
     \Vertex[L=$s_1$, x=0, y=0]{B}
   \Vertex[L=$s_0$, x=-3,y=0]{A}
   \Vertex[L=$s_2$, x=3, y=0]{C}
   \Edge[label = a/x, labelstyle = above](A)(B)
   \Edge[label = b/y ,labelstyle = above](B)(C)
   \Edge[label = b/y, labelstyle = below](C)(B)
   \Edge[label = a/x ,labelstyle = below](B)(A)
   \Loop[dist = 2cm, dir = NO, label = b/y ,labelstyle = left](A.west)
   \Loop[dist = 2cm, dir = SO, label = a/x, labelstyle = right](C.east)
 \end{tikzpicture}}
 \caption{A complete automaton $M$ on $A=\{a,b,x=a^{-1},y=b^{-1}\}$.}
\label{fig_aut2}
\end{figure}
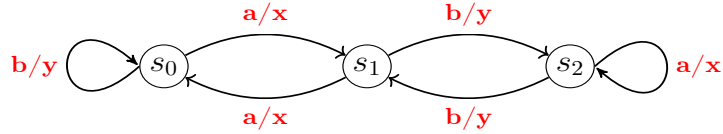
Once a unique start state and a unique accept state are specified, the automaton $M$ from Figure \ref{fig_aut2} is  bi-deterministic. We describe the consequences of several different choices on  $L(M)$, the  language accepted by $M$. If $s_0$ is both the start and accept state of $M$, then $L(M)$ is   the subgroup $H$ of $A^*$ generated by $\{b,a^2,ab^2a,ababa\}$. If $s_1$ is both the start and accept state of $M$, then $L(M)$ is   the subgroup $a^{-1}Ha$ of $A^*$ and if $s_2$ is both the start and accept state of $M$, then $L(M)$ is   the subgroup $(ab)^{-1}H(ab)$ of $A^*$. If $s_0$ is the start state and $s_1$ the  accept state, then $L(M)$ is  the coset $Ha$ of $H$. If $s_0$ is the start state and $s_2$ the accept state, then $L(M)$ is  the coset $Hab$ of $H$. When $s_0$ is both the start and accept state of $M$, $M$ is called \emph{the Schreier automaton for $F_2$ relative to the subgroup $H$} \cite[p.102]{sims}. 
Generally,  the Schreier automaton for $F_2$ relative to a subgroup $H$ is a finite,  complete and bi-deterministic automaton, with  transition monoid $T$  a group of permutations in $S_d$,  with  $T \simeq\,^{G}\big/_{N_H}$, where $N_H= \bigcap\limits^{}_{g \in G}g^{-1}Hg$ is the normal core of $H$. \end{ex}
\section{Covering spaces  in the context of free groups}
 \subsection{How to look at the covering space of  the $n$-leaves bouquet ?} \label{subsec-covering-free}
We refer to \cite[Chapter 2]{stilwell}, and \cite{ivanov} for details. Let $X$ denote  the  $n$-leaves  bouquet  with basepoint $x_0$. Its fundamental group $\pi_1(X,x_0)$ is $F_n$. Given $H \leq F_n$ of index $d$,  there exists  a covering $(\tilde{X}_H,p)$ of $X$ with basepoint $\tilde{x}_0$ (see Theorem \ref{thm_cov_corresp}).  The covering $(\tilde{X}_H,p)$ can be viewed as a directed and labelled graph.  Using this covering space, J. Stallings  gave  a topological proof of some classical results about finitely generated subgroups of $F_n$. In \cite{stallings}, he introduced the topological  notion of a folding, and  used it  in the study of finitely generated subgroups of a free group.\\

 \setlength\parindent{10pt}The folding corresponding to a finitely generated subgroup $H$ of $F_n$ is a directed labelled graph, that can be  considered as  a finite deterministic automaton; fixing the start and the end state at a distinguished vertex, it recognises the set of elements in $H$. Stalling's approach has been applied to solve many combinatorial and algorithmic problems in group theory. In the years 1980-90, similar graphical approaches were developed in the study of finitely generated  subgroups of the free group, in order to solve many special cases of the Hanna Neumann conjecture \cite{gersten,nickolas,tardos}, which  states: $rank(H \cap K)-1 \leq (rank(H)-1)(rank(K)-1)$, where $H$ and $K$  are finitely generated subgroups of the free group (of ranks $rank(H)$ and $rank(K)$ respectively). It was recently solved  \cite{mineyev1}, \cite{friedman}, \cite{dicks}.

 \setlength\parindent{10pt}In this work, our approach is the following:  given $H \leq F_n$ of index $d$,  we consider the  covering $(\tilde{X}_H,p)$ of $X$ with basepoint $\tilde{x}_0$, from several points of view in parallel: as a covering, as a  Schreier coset diagram  for $F_n$ relative to the subgroup  $H$, as a  complete and "almost" bi-deterministic automaton, in the sense that we do not fix the  start and accept states in advance, but we fix them differently at our needs. We will use all these points of view together. More formally:  
 \begin{defn}\label{def_Schreier-graph}
 Let $H < F_n$ be of finite index $d$. Let $(\tilde{X}_H,p)$ be the covering  of the $n$-leaves bouquet with basepoint $\tilde{x}_0$ and  vertices  $\tilde{x}_0, \tilde{x}_1,...,\tilde{x}_{d-1}$. Let  $t_i$ denote the label of a minimal path from $\tilde{x}_0$ to $\tilde{x}_i$. Let $\mathscr{T}=\{1, t_i\mid 1 \leq i\leq d-1\}$.  As $\tilde{X}_H$ is the Schreier coset diagram  for $F_n$ relative to the subgroup  $H$,  $\tilde{x}_0$ represents the subgroup $H$ and the other vertices $\tilde{x}_1,...,\tilde{x}_{d-1}$ represent the cosets  $Ht_i$ accordingly.  We call $\tilde{X}_{H}$  \emph{the  Schreier graph  of $H$},  with this correspondence between the vertices  $\tilde{x}_0, \tilde{x}_1,...,\tilde{x}_{d-1}$ and the cosets  $Ht_i$ accordingly.  
   \end{defn} 
  Note that  a minimal path from $\tilde{x}_0$ to $\tilde{x}_i$ is not necessarily unique, so we choose arbitrarily one. The transversal  $\mathscr{T}$  may be a Schreier transversal or any transversal with representatives of minimal length.
\begin{ex}\label{ex_schreier-coset}
Consider the automaton in Figure \ref{fig_aut2}.  Replacing $s_0$ by $G$ and the other vertices accordingly, this graph is the  Schreier graph of  $G=\langle b,a^2,ab^2a,ababa\rangle$,   a subgroup of $F_2$ of index 3, with Schreier transversal  $\mathscr{T}=\{1,a,ab\}$: 
\begin{figure}[H]
\centering \scalebox{0.9}[0.8]{\begin{tikzpicture}
  \SetGraphUnit{4}
   \tikzset{VertexStyle/.append  style={fill}}
    \Vertex[L=$Ga$, x=0, y=0]{B}
  \Vertex[L=$G$, x=-3,y=0]{A}
  \Vertex[L=$Gab$, x=3, y=0]{C}
  \Edge[label = a, labelstyle = above](A)(B)
  \Edge[label = b, labelstyle = above](B)(C)
  \Edge[label = b, labelstyle = below](C)(B)
  \Edge[label = a, labelstyle = below](B)(A)
  \Loop[dist = 2cm, dir = NO, label = b, labelstyle = left](A.west)
  \Loop[dist = 2cm, dir = SO, label = a, labelstyle = right](C.east)
\end{tikzpicture}}
\caption{The Schreier graph of  $G=\langle b,a^2,ab^2a,(ab)^3\rangle$.}\label{Schreier-G}
\end{figure}
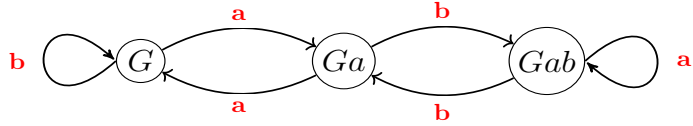\end{ex}
\flushleft \begin{rem}
If  $H \leq F_n$ of index $d$ with Schreier  graph $\tilde{X}_H$,   then  from Theorem \ref{thm_cov_corresp} (covering theory),  the Deck group $\operatorname{Cov}(\tilde{X}/X)\simeq\,^{N(H)}\big/_{H}$, where $N(H)$ is the normaliser of $H$ in  $F_n$. From Section \ref{subsec_automat} (automaton theory), the transition group $T \simeq \,^{F_n}\big/_{N_H}$, where $N_H$ is the normal core of $H$. So, whenever $H$ is a normal subgroup of $F_n$, $\operatorname{Cov}(\tilde{X}/X)\simeq T$ and it can be read easily from the graph. If $H$ is not normal in $F_n$, this is not necessarily true. Indeed, in Example \ref{ex_schreier-coset}, $G$  is of index 3,  not normal, so  $\operatorname{Cov}(\tilde{X}/X)=\{1\}$, while  its transition group is $S_3$ (see  Example \ref{ex_automaton_index3}).
\end{rem}
 \subsection{A combinatorial approach to coverings of the leaves bouquet}\label{subsec_schreier-graph}
We keep the same terminology as  in Section \ref{subsec-covering-free}. We develop a combinatorial approach to $\tilde{X}_H$ and prove some of its properties.
 If $H \leq F_n$ of index $d$, then for every $w \in F_n$, there exists a minimal natural number $1 \leq m_{w,i}\leq d$ such that $w^{m_{w,i}} \in t_i^{-1}Ht_i$, where $t_i \in \mathscr{T}$.
  \begin{lem}\label{lem_exists_order2}
 Let $H \leq F_n$  of index $d$.  Let $\tilde{X}_H$ be the Schreier graph of $H$ with vertices  $\tilde{x}_0, \tilde{x}_1,...,\tilde{x}_{d-1}$.  Then, for every $w \in F_n$,  there exists a minimal natural number $1 \leq m_{w,i}\leq d$ such that $w^{m_{w,i}}$ is a loop at $\tilde{x}_i$.
 \end{lem}
  \begin{proof}
Since  $1 \leq m_{w,i}\leq d$ is the  minimal natural number such that $w^{m_{w,i}} \in t_i^{-1}Ht_i$,  $ m_{w,i}$ is also the minimal natural number such that $w^{m_{w,i}}$ is a loop at $\tilde{x}_i$.
  \end{proof}
  In \cite{thesis}, the authors define the order of an element in regular coverings. We extend the definition of  the order of an element to any  covering.
\begin{defn}\label{defn_order}
 Let $H < F_n$ of  index $d$. Let  $w \in F_n$.  Let $\tilde{X}_H$,  $\mathscr{T}$ as  in Definition \ref{def_Schreier-graph}. 
\begin{enumerate}[(i)]
\item \emph{The order of $w$ at $\tilde{x}_i$ in $\tilde{X}_H$} is the   minimal natural number $1 \leq o(w,i) \leq d$ such that $w^{o(w,i)}$ is  a loop at $\tilde{x}_i$,  or equivalently such that $w^{o(w,i)} \in t_i^{-1}Ht_i$.
\item \emph{The $k$-th $w$-step}, $1 \leq k \leq o(w,i)$, is the $k$-th path labelled $w$ in the loop $w^{o(w,i)}$ at $\tilde{x}_i$.
\item The set $V_{w,i}$ is the set of vertices in  $\tilde{X}_H$ visited at each $w$-step in the loop $w^{o(w,i)}$ at $\tilde{x}_i$.
\item The loops of $w$ at $\tilde{x}_i$ and $\tilde{x}_j$ are \emph{disjoint} if $V_{w,i} \cap V_{w,j}=\emptyset$
\end{enumerate}
\end{defn}
\begin{ex}\label{ex_compute_o_and_V}
Consider Examples \ref{ex_schreier-coset}.  For $w=aba^{-1}$, $o(aba^{-1},0)=2$, $V_{aba^{-1},0}=V_{aba^{-1},2}=\{\tilde{x}_0,\tilde{x}_2\}$ and 
$o(aba^{-1},1)=1$, $V_{aba^{-1},1}=\{\tilde{x}_1\}$,   that is the loops at $\tilde{x}_0$ and $\tilde{x}_1$ are disjoint. If we reduce cyclically $aba^{-1}$, we get   $b$, with   $o(b,0)=1$, $V_{b,0}=\{\tilde{x}_0\}$; $o(b,1)=o(b,2)=2$, $V_{b,1}=V_{b,2}=\{\tilde{x}_1,\tilde{x}_2\}$. If $w=ab$, $o(ab,0)=o(ab,1)=o(ab,2)=3$,  $V_{a,0}=V_{a,1}=V_{a,2}=\{\tilde{x}_0,\tilde{x}_1,\tilde{x}_2\}$.
\end{ex}
 \begin{lem}\label{lem_o(w,i)}
 Let $H < F_n$  of index $d$. Let $N_H$ denote the normal core of $H$. Let $\tilde{X}_H$ be the Schreier graph of $H$ with vertices  $\tilde{x}_0, \tilde{x}_1,...,\tilde{x}_{d-1}$.  Let $w \in F_n$ with $o(w,i)$ and $V_{w,i}$, $0 \leq i\leq d-1$, as in Definition \ref{defn_order}. Then
\begin{enumerate}[(i)]
\item  $\mid  V_{w,i} \mid=o(w,i)$. 
\item Either  $V_{w,i} \cap V_{w,j}=\emptyset$, or $V_{w,i} = V_{w,j}$, $0 \leq i,j\leq d-1$.
\item $\bigcup\limits^{i=d-1}_{i=0}V_{w,i}=\{\tilde{x}_0,..,\tilde{x}_{d-1}\}$.
\item $\sum\limits^{}_{}o(w,i)=d$, where the sum is on disjoint loops.
\item If for some integer $k$,  $w^k$ is a loop at $\tilde{x}_i$,   then  $o(w,i)$ divides $k$.
\item  If $o(w,i)=d$ for some $0 \leq i \leq d-1$, then $w^d \in N_H$. 
\item If  $w^k \in N_H$ for  $k \in \mathbb{Z}$, then  $\ell cm\{o(w,i)\mid 0 \leq i \leq d-1\}$ divides $k$.
\item   $o(w,i)$ divides  $[F_n:  N_H]$. 
\end{enumerate}
In particular, $w \in N_H$ if and only if $o(w,i)=1$, for all $0 \leq i\leq d-1$.
\end{lem}
\begin{proof}
$(i)$ Clearly,  $\mid  V_{w,i} \mid \leq o(w,i)$.  If   $\mid  V_{w,i} \mid < o(w,i)$, then the loop crosses a vertex at least twice and there exists a loop $w^{j}$ at $\tilde{x}_i$ with $j<o(w,i)$, so this contradicts the minimality of  $o(w,i)$.\\ $(ii)$ Assume  by contradiction there exists $\tilde{x}_k \in V_{w,i} \cap V_{w,j}$,  with $V_{w,i} \neq V_{w,j}$.  So, from $(i)$, $o(w,k)=o(w,i)=o(w,j)$, and there are two different loops of the same length $o(w,i)$ at $\tilde{x}_k$. So, beginning at $\tilde{x}_k$, there exists  a first $j < o(w,i)$, such that at the $j$-th $w$-step in one loop the vertex $\tilde{x}_{k'}$ is attained and  in the other loop a different vertex $\tilde{x}_{k''}$ is attained. That is, the element $w^j$ acts in two different ways on the coset of $H$ represented by $\tilde{x}_k$ and this is a contradiction.\\
$(iii),(iv)$ The sets $V_{w,i}$ partition the finite set $\{\tilde{x}_0,..,\tilde{x}_{d-1}\}$.\\
$(v)$ From the minimality of $o(w,i)$, $o(w,i)\leq k$. We divide $k$ by $o(w,i)$: $k= q(o(w,i))+r$, with residue  $0\leq r< o(w,i)$.  So, $w^r=w^{k-q(o(w,i))}$ is a  loop at $\tilde{x}_i$. From the minimality of $o(w,i)$, $r=0$, that is $o(w,i)$ divides $k$. \\ 
$(vi)$ If  $o(w,i)=d$, then $w^{d}$ is a  loop at all the $d$ vertices. So,  $w^d$ belongs to all the conjugates of $H$, that is $w^d\in N_H$. \\
$(vii)$, $(viii)$ If  $w^k \in N_H$, then  from $(v)$, $o(w,i)$ divides $k$ for every $0 \leq i \leq d-1$, that is  $\ell cm\{o(w,i)\mid 0 \leq i \leq d-1\}$ divides $k$. As $N_H\lhd F_n$, $k$ divides  $[F_n:  N_H]$ and so $o(w,i)$ divides  $[F_n:  N_H]$, for all $0 \leq i \leq d-1$. 
\end{proof}
Note that from Lemma \ref{lem_o(w,i)}$(v)$, it results that for every integer $q$ dividing  $o(w,i)$, there exists an element $g \in F_n$ with  $o(g,i)=q$ (take $g=w^{\frac{o(w,i)}{q}}$).
We now consider the  case of a normal subgroup, with a regular covering.
\begin{lem}\label{lem_o(w,i)_normal}
Let $N \lhd F_n$ of  index $m$. Let  $\tilde{X}_N$ be   the Schreier graph of $N$ with vertices  $\tilde{x}_0, \tilde{x}_1,...,\tilde{x}_{m-1}$.  Let $\mathscr{T}=\{1, t_i\mid 1 \leq i\leq m-1\}$ be a transversal of $N$ in $F_n$, with  $t_i$  the label of a minimal path from $\tilde{x}_0$ to $\tilde{x}_i$. Let $w \in F_n$ with $o(w,i)$ and $V_{w,i}$, $0 \leq i\leq m-1$, as in Definition \ref{defn_order}. Then
\begin{enumerate}[(i)]
\item  $o(w,i) =o(w,j)=o(w)$, for all $0 \leq i,j \leq m-1$. 
\item  $o(w)$ divides $m$.
\item There are exactly $\frac{m}{o(w)}$ disjoint loops labelled $w^{o{(w)}}$ in  $\tilde{X}_N$.
\end{enumerate}
\end{lem}
\begin{proof}
$(i)$ We recall that $o(w,i)$ is the   minimal natural number  such that $w^{o(w,i)} \in t_i^{-1}Nt_i$. Since  $N \lhd F_n$, $g^{-1}Ng=N$ for every $g \in F_n$.  So, for all $0 \leq i,j \leq d-1$,  $o(w,i) =o(w,j)=o(w)$.\\
$(ii)$ results from Lemma \ref{lem_o(w,i)}$(iv)$ and $(iii)$  from Lemma \ref{lem_o(w,i)}$(iii)$. 
\end{proof}
We introduce a graph derived from the Schreier graph that describes the $w$-steps, where $w$ is a word in $F_n$.
\begin{defn}\label{def_X_H(w)}
 Let  $w \in F_n$.  Let $H<F_n$ of index $d$. Let  $\tilde{X}_H$ be the 
 Schreier graph  of  $H$.   We define $\tilde{X}_{H}(w)$ to be the  Schreier graph  of $H$, but instead of being labelled by the generators of $F_n$, only the $w$-steps are specified. We also call   $\tilde{X}_{H}(w)$  the  Schreier graph  of $H$.
  \end{defn} 
  \begin{ex}
  Consider the subgroup $G$ of index $3$ in $F_2$ from Example \ref{ex_schreier-coset}. Let $w=ab$. Then, $\tilde{X}_G(ab)$ is the following connected graph: 
   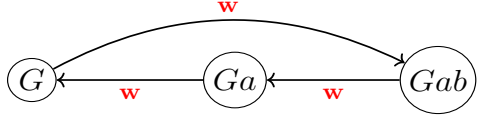
\begin{figure}[H]
    \centering \scalebox{0.9}[0.8]{\begin{tikzpicture} 
          \SetGraphUnit{4}
           \tikzset{VertexStyle/.append  style={fill}}
            \Vertex[L=$Ga$, x=0, y=0]{B}
          \Vertex[L=$G$, x=-3,y=0]{A}
          \Vertex[L=$Gab$, x=3, y=0]{C}
          \Edge[label = w, labelstyle = above](A)(C)
           \tikzset{EdgeStyle/.style = {->}}
          \Edge[label = w, labelstyle = below](B)(A)
          \Edge[label = w, labelstyle = below](C)(B)
               \end{tikzpicture}}        
        \caption{The graph $\tilde{X}_G(ab)$ for  $G=\langle b,a^2,ab^2a,(ab)^2a\rangle$.} \label{fig-X_G(ab)}
         \end{figure}
          Let $w=a$. The graph $\tilde{X}_G(a)$ is the following disconnected graph: 
  \begin{figure}[H]
  
  \centering \scalebox{0.9}[0.8]{ \begin{tikzpicture} 
    \SetGraphUnit{4}
     \tikzset{VertexStyle/.append  style={fill}}
      \Vertex[L=$Ga$, x=0, y=0]{B}
    \Vertex[L=$G$, x=-3,y=0]{A}
    \Vertex[L=$Gab$, x=3, y=0]{C}
    \Edge[label = a, labelstyle = above](A)(B)
    \Edge[label = a, labelstyle = below](B)(A)
     \Loop[dist = 2cm, dir = SO, label = a, labelstyle = right](C.east)
  \end{tikzpicture}}  
  \caption{The graph $\tilde{X}_G(a)$ for $G=\langle b,a^2,ab^2a,(ab)^2a\rangle$.} \label{fig-X_G(a)}
   \end{figure}
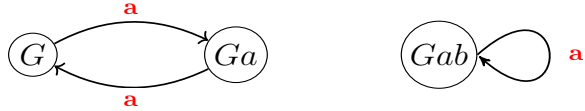 \end{ex}
From Lemma \ref{lem_o(w,i)_normal}, the situation described in Figure \ref{fig-X_G(a)} can never occur if the subgroup  is normal.
\subsection{What happens when we look at a tower of  covering spaces?}\label{sec_combined}
Let $H < F_n$  of index $d$.  Let $\tilde{X}_H$ be the Schreier graph of $H$.   Let $N \lhd F_n$  of  index $m$ such that $N \subset H$, with Schreier graph  $\tilde{X}_N$. The subgroup $H$ is a free group of rank  $d(n-1)+1$, so it is the fundamental group of  $X_H$, the bouquet of $d(n-1)+1$ circles.
 As  $N < H$, there exists a regular $\frac{m}{d}$-sheeted covering $\hat{X}_N$ over  $X_H$, such that the  $\frac{m}{d}$ vertices in  $\hat{X}_N$ represent the  $\frac{m}{d}$ right cosets of $N$ whose disjoint union is $H$ (see Theorem \ref{thm_cov_inject}). We now introduce  a graph  that combines in some sense the Schreier graphs  of $N$ and $H$, and gives a more complete picture. 
 \begin{defn}\label{def_combined}
 Let  $w \in F_n$.  Let $H<F_n$ of index $d$ and let $N\lhd F_n$ of index $m$ such that $N\subset H$.  Let  $\tilde{X}_H(w)$, $\tilde{X}_N(w)$  as in Definition \ref{def_X_H(w)}. We define \emph{the combined graph of $N$ and $H$}, denoted by $\tilde{X}_{N,H}(w)$, to be  the following two-levelled graph:\\
 At  the top layer, denoted  \emph{Top}: $\tilde{X}_{N}(w)$ \\
  At the bottom layer, denoted  \emph{Bot}: $\tilde{X}_{H}(w)$\\
  The relative position of the vertices in \emph{Top} and \emph{Bot} is:\\ For each vertex $\tilde{x}_i$ representing $Ht_i$ in \emph{Bot}, there is above it  a set of $\frac{m}{d}$ vertices $Ny_1$,..,$Ny_{\frac{m}{d}}$ such that $Ht_i$ is the disjoint union of $Ny_1$,..,$Ny_{\frac{m}{d}}$. 
  We call  the set $\{Ny_1,..,Ny_{\frac{m}{d}}\}$  \emph{the fiber over  $Ht_i$}, denoted by $p^{-1}_{w}(Ht_i)$.
  \end{defn}  
 As $N\subset H$, the action of $F_n$ on the right cosets of $H$ induces an action of $F_n$ on the right cosets of $N$. That is, 
  for each loop in \emph{Bot}, there is at least one loop induced in \emph{Top}. Indeed, consider the loop labelled $w^{o(w,i)}$ at $\tilde{x}_i$ in  \emph{Bot} with $V_{w,i}$, $0 \leq i\leq d-1$, as in Definition \ref{defn_order}.  Then, if  we denote by $p^{-1}_w(V_{w,i})$ the set of  fibers  in \emph{Top} over  the set of visited vertices, then there are some loops connecting the fibers in  $p^{-1}_w(V_{w,i})$. We will study these loops in more detail in the following lemmas. 
    \begin{ex}\label{ex_combined_graph}
 Let $K<F_2$ of index $4$ from  Example \ref{ex_cov_subgrps}. Let $N=N_K$ of index $8$ in $F_2$, with $\tilde{X}_N$ below,  and $\tilde{X}_{N,K}(ab)$ in Figure \ref{fig-combined N-K}: 
\begin{figure}[H]
 \centering \scalebox{0.8}[0.7]{\begin{tikzpicture}\label{fig_covering_normalcoreK}
    \SetGraphUnit{1}
     \tikzset{VertexStyle/.append  style={fill}}
      \Vertex[L=$N$, x=-8,y=0]{A}
      \Vertex[L=$Na$, x=-6, y=0]{B}
    
    \Vertex[L=\small$Nab$, x=-4, y=0]{C}
     \Vertex[L=\small$Naba$, x=-2, y=0]{D}
      \Vertex[L=\tiny $N(ab)^2$, x=0, y=0]{E}
       \Vertex[L=\small$Nbab$, x=2, y=0]{F}
  \Vertex[L=\small$Nba$, x=4, y=0]{G}
   \Vertex[L=\small$Nb$, x=6, y=0]{H}
    \Edge[label = a, labelstyle = above right](A)(B)
     \Edge[label =a, labelstyle = below right](B)(A)
    \Edge[label = b, labelstyle = above](B)(C)
    \Edge[label = b, labelstyle = below](C)(B)
     \Edge[label = a, labelstyle = above](C)(D)
    \Edge[label = a, labelstyle = below](D)(C)
    \Edge[label = b, labelstyle = above](D)(E)
      \Edge[label = b, labelstyle = below](E)(D)
       \Edge[label = a, labelstyle = above](E)(F)
         \Edge[label =a, labelstyle = below](F)(E)
          \Edge[label = b, labelstyle = above](A)(H)
            \Edge[label =b, labelstyle = below](H)(A)
             \Edge[label =b, labelstyle = above](F)(G)
                \Edge[label =b, labelstyle = below](G)(F)
                \Edge[label = \tiny a, labelstyle = above left](G)(H)
   \Edge[label =\tiny a, labelstyle = below left](H)(G)
    \end{tikzpicture}}
     \end{figure}
    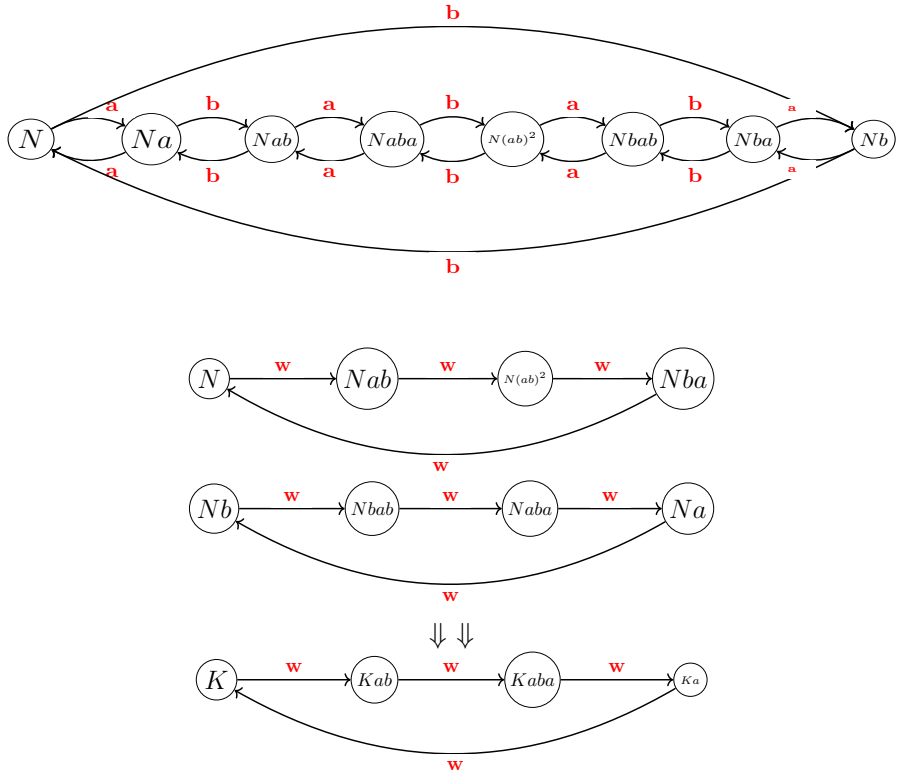
\begin{figure}[H] 
        \centering  \scalebox{0.7}{\begin{tikzpicture}
        \SetGraphUnit{4}
         \tikzset{VertexStyle/.append  style={fill}}
          \Vertex[L=$N$, x=-3,y=0]{A}
          \Vertex[L=$Nab$, x=0, y=0]{B}
        
        \Vertex[L=\tiny $N(ab)^2$, x=3, y=0]{C}
         \Vertex[L=$Nba$, x=6, y=0]{D}
        
        \Edge[label = w, labelstyle = below](D)(A)
     \tikzset{EdgeStyle/.style = {->}}   
      \Edge[label = w, labelstyle = above](A)(B)
              \Edge[label = w, labelstyle = above](B)(C)
              \Edge[label = w, labelstyle = above](C)(D)
      \end{tikzpicture}}\\
       \scalebox{0.7}{\begin{tikzpicture}
              \SetGraphUnit{4}
               \tikzset{VertexStyle/.append  style={fill}}
                \Vertex[L=$Nb$, x=-3,y=0]{A}
                \Vertex[L=\small$Nbab$, x=0, y=0]{B}
              
              \Vertex[L=\small $Naba$, x=3, y=0]{C}
               \Vertex[L=$Na$, x=6, y=0]{D}
             
              \Edge[label = w, labelstyle = below](D)(A)
         \tikzset{EdgeStyle/.style = {->}}      
             \Edge[label = w, labelstyle = above](A)(B)
                          \Edge[label = w, labelstyle = above](B)(C)
                          \Edge[label = w, labelstyle = above](C)(D)
            \end{tikzpicture}}\\
  \centering$\Downarrow$ $\Downarrow$\\
     
   \scalebox{0.7}{\begin{tikzpicture}
                \SetGraphUnit{4}
                 \tikzset{VertexStyle/.append  style={fill}}
                  \Vertex[L=$K$, x=-3,y=0]{A}
                  \Vertex[L=\small$Kab$, x=0, y=0]{B}
                
                \Vertex[L=\small $Kaba$, x=3, y=0]{C}
                 \Vertex[L=\tiny $Ka$, x=6, y=0]{D}
               
                \Edge[label = w, labelstyle = below](D)(A)
           \tikzset{EdgeStyle/.style = {->}}      
               \Edge[label = w, labelstyle = above](A)(B)
                            \Edge[label = w, labelstyle = above](B)(C)
                            \Edge[label = w, labelstyle = above](C)(D)
              \end{tikzpicture}}
              \caption{The combined graph $\tilde{X}_{N,K}(w)$:  \emph{Top}:$\tilde{X}_{N}(w)$, and \emph{Bot}:$\tilde{X}_{K}(w)$, with  arrows  delimiting between \emph{Top} and \emph{Bot}, $w=ab$.}  \label{fig-combined N-K}
    \end{figure}   \flushleft
  In \emph{Bot}, there is one  loop of length $4$, labelled $w^4$, since $o_K(w)=4$.  In \emph{Top}, there are $2$ disjoint  loops of length $4$, labelled $w^4$.  The fiber of $K$ is  the set $p^{-1}_w(K)=\{N,Nb\}$, with all its  elements just above $K$, the fiber of $Kab$ is  $p^{-1}_w(Kab)=\{Nab,Nbab\}$ with all its  elements just above $Kab$ and so on. 
 Now, consider  $H=\langle b,a^2,aba\rangle <F_2$ of index 2. Then $N<H$ and $\tilde{X}_{N,H}(ab)$ is:
 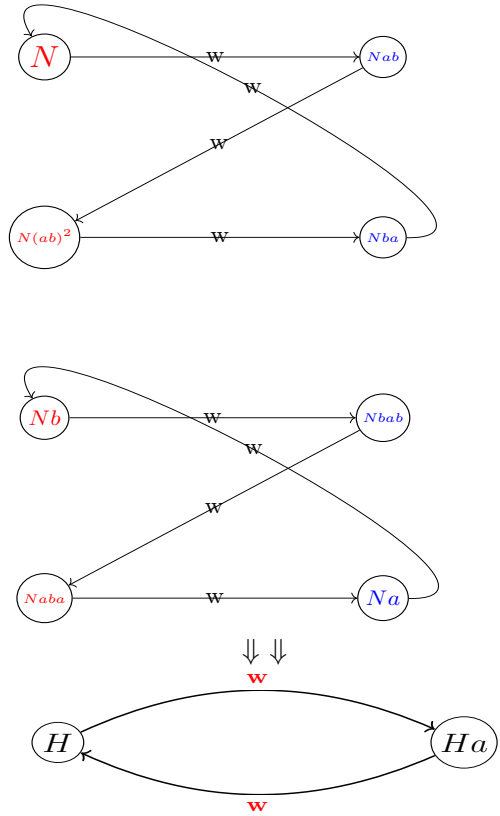
\begin{figure}[H]
\centering \scalebox{0.9}[0.8]{\begin{tikzpicture}
        \SetGraphUnit{10}
       \tikzset{VertexStyle/.append  style={fill}}
     
     \Vertex[L=\color{red}{$N$}, x=-5, y=3]{B}
         \Vertex[L=\tiny\color{red}{$N(ab)^2$}, x=-5,y=0]{A}
          \Vertex[L=\small\color{red}{$Nb$}, x=-5,y=-3]{X}
           \Vertex[L=\tiny\color{red}{$Naba$}, x=-5,y=-6]{Y}
        \SetGraphUnit{4}
         \tikzset{VertexStyle/.append  style={fill}}
          \tikzset{every loop/.style={min distance=10mm,looseness=10}}
         \Vertex[L=\tiny\color{blue}{$Nab$}, x=0, y=3]{D}
           \Vertex[L=\tiny\color{blue}{$Nba$}, x=0,y=0]{C}
        \Vertex[L=\tiny\color{blue}{$Nbab$}, x=0, y=-3]{Z}
 \Vertex[L=\small\color{blue}{$Na$}, x=0,y=-6]{W}

        \path[->]          (B)  edge                 node {w} (D);
          \path[->]          (D)  edge      node {w} (A);
         \path[->]          (A)  edge     node {w} (C);
          \path[->]          (C)  edge        [bend right=150]          node [swap] {w} (B);
            
                  \path[->]          (X)  edge                 node {w} (Z);
                    \path[->]          (Z)  edge      node {w} (Y);
                   \path[->]          (Y)  edge     node {w} (W);
                    \path[->]          (W)  edge        [bend right=150]          node [swap] {w} (X);
            \end{tikzpicture}}\\
       \centering$\Downarrow$ $\Downarrow$\\
 \centering   \scalebox{0.9}[0.7]{\begin{tikzpicture}
   \SetGraphUnit{4}
    \tikzset{VertexStyle/.append  style={fill}}

    \Vertex[L=$Ha$, x=-1, y=0]{B}
   \Vertex[L=$H$, x=-7,y=0]{A}
   
   \Edge[label = w, labelstyle = above](A)(B)

   \Edge[label = w, labelstyle = below](B)(A)
   \end{tikzpicture}}
 \caption{The combined graph $\tilde{X}_{N,H}(w)$:  \emph{Top}:$\tilde{X}_{N}(w)$, and \emph{Bot}:$\tilde{X}_{H}(w)$, with  arrows  delimiting between \emph{Top} and \emph{Bot}, $w=ab$.} \label{fig_graph_combineN_H}
 \end{figure}
  \flushleft
 The fiber of $H$ is  $p^{-1}_w(H)=\{N,N(ab)^2,Nb,Naba\}$, with all its  elements just above $H$  and the fiber of $Ha$ is $p^{-1}_w(Ha)=\{Nab,Nba,Nbab,Na\}$,  with all its  elements just above $Ha$.  In \emph{Bot}, there is one  loop of length $2$, labelled $w^2$, since $o_H(w)=2$. In \emph{Top}, there are $2$ disjoint  loops of length $4$, labelled $w^4$.  These are  exactly the same loops as in Figure \ref{fig-combined N-K}, but their distribution into the fibers are different. Here each fiber contributes two elements in a loop, while in \emph{Top} of $\tilde{X}_{N,K}(ab)$  each fiber contributes only one  element in a loop.
 \end{ex} 
  The next lemma describes some basic properties of the  combined  graph that are derived directly from its definition and the properties of $\tilde{X}_{H}(w)$, $\tilde{X}_{N}(w)$ (see Section \ref{subsec_schreier-graph}).
 \begin{lem}\label{lem_basic_combined}
  Let  $w \in F_n$.
 Let $H < F_n$ of index $d$ and  $N \lhd F_n$ of index $m$ such that $N \subset H$. Let   $\tilde{X}_{N,H}(w)$ be their combined  graph.  Then
 \begin{enumerate}[(i)]
  \item In Bot, there are  $d$ vertices.
  \item In Top, there are  $d$ fibers each of size  $\frac{m}{d}$.
 \item In  Top,  the  fiber over $Hx$ is  the set of  $\frac{m}{d}$ right cosets of $N$ whose disjoint union is $Hx$.
  \item In Top, there are exactly $\frac{m}{o_N(w)}$ disjoint loops labelled $w^{o_N(w)}$.
 \end{enumerate}
 \end{lem}
In the following lemmas, we describe the connection between the top layer and the bottom layer in the combined graph. The next lemma is very important for the rest of the paper, as  here we describe  entirely  the induced  loops in \emph{Top}. 
  \begin{lem}\label{lem_number_loops}
   Let  $w \in F_n$.
  Let $H < F_n$ of index $d$ and  $N \lhd F_n$ of index $m$ such that $N \subset H$. Let   $\tilde{X}_{N,H}(w)$ be their combined  graph, with  
 $\tilde{x}_0, \tilde{x}_1,...,\tilde{x}_{d-1}$ the vertices in \emph{Bot}. 
 Let $o(w,j)$ be the order of $w$  at a chosen vertex $\tilde{x}_j$ in  \emph{Bot}  and let $V_{w,j}$ be  the set of vertices visited at each $w$-step in this  loop.  Let $p_w^{-1}(V_{w,j})$ be the set of fibers over $V_{w,j}$ in \emph{Top}. Then
  \begin{enumerate}[(i)]
  \item $\mid p_w^{-1}(V_{w,j}) \mid= \mid V_{w,j} \mid =o(w,j)$.
   
  \item The loop in $\emph{Bot}$ labelled $w^{o(w,j)}$ between the vertices in $V_{w,j}$  induces $\frac{\frac{m}{d}}{ \frac{o_N(w)}{o(w,j)}}$
  loops labelled $w^{o_N(w)}$ between the $o(w,j)$ fibers in  $p^{-1}(V_{w,j})$ in $\emph{Top}$.
  \item  Each fiber in  $p_w^{-1}(V_{w,j})$ contributes exactly  $\frac{o_N(w)}{o(w,j)}$ elements in an induced  loop labelled $w^{o_N(w)}$.
  
  \end{enumerate}
  \end{lem}
 \begin{proof}
 $(i)$  As $N\subset H$, the action of $F_n$ on the right cosets of $H$ induces an action of $F_n$ on the corresponding right cosets of $N$.\\
 $(ii),(iii)$  To simplify the notation, let $k=o(w,j)$, let $\ell$ denote the   loop in $\emph{Bot}$ labelled $w^k$  at $\tilde{x}_j$. Assume  $\tilde{x}_j$ represents the right coset $Hx$ and $p^{-1}_w(Hx)=\{Ny_1,Ny_2,...,Ny_{\frac{m}{d}}\}$.  Each  $w$-step in $\emph{Bot}$ between two vertices $Hx$ and $Hxw$  induces a $w$-step in $\emph{Top}$ between   $Ny_i$ and $Ny_iw$, for  every $1 \leq i \leq \frac{m}{d}$. 
  So, $\ell$  induces 
  $Ny_i\rightarrow Ny_iw \rightarrow Ny_iw^2\rightarrow ...\rightarrow Ny_iw^{k-1}\rightarrow Ny_iw^{k}\rightarrow Ny_iw^{k+1}...$. From  Definition \ref{defn_order}$(i)$, $k$ is the minimal natural number such that $w^k \in x^{-1}Hx$ or equivalently $Hxw^k=Hx$. So, 
  $Ny_iw^{k} \in p^{-1}_w(Hx)$, $Ny_iw^{k+1} \in p^{-1}_w(Hxw)$ and so on.\\
   If $k=o_N(w)$, then  $Ny_iw^{k}=Ny_i$, for all  $ 1 \leq i \leq \frac{m}{d}$ and  each fiber in  $p^{-1}(V_{w,j})$ contributes exactly  $\frac{o_N(w)}{o(w,j)}=\frac{o_N(w)}{k}=1$ element in a loop labelled $w^{o_N(w)}$.  Since there are $\frac{m}{d}$ elements in each fiber, the loop in $\emph{Bot}$ labelled $w^{k}$ induces $\frac{\frac{m}{d}}{ \frac{o_N(w)}{k}}= \frac{m}{d}$
  loops labelled $w^{o_N(w)}$ between the $o(w,j)$ fibers in  $p^{-1}(V_{w,j})$.\\
  Assume now that $\frac{o_N(w)}{k}=p>1$. Then  $Ny_iw^{k} \in p^{-1}_w(Hx)$, but $Ny_iw^{k} \neq Ny_i$ and  $Ny_iw^{pk}=Ny_i$. In fact, $Ny_iw^{k} = Ny_{i_1}$,  $Ny_iw^{2k} = Ny_{i_2}$,..., $Ny_iw^{(p-1)k} = Ny_{i_{p-1}}$ for different $i_1,i_2,...,i_{p-1}$ between $1$ and $\frac{m}{d}$. So, $\ell$ induces a loop of length $o_N(w)$ between the $k$ fibers of the form:\\
 $\begin{array}{ccccccc}
 Ny_i  &\rightarrow & Ny_iw &\rightarrow & Ny_iw^2 &\rightarrow ... & Ny_iw^{k -1}\rightarrow\\ 
 Ny_iw^{k} & \rightarrow & Ny_iw^{k+1} & \rightarrow & Ny_iw^{k+2} & \rightarrow ...&  Ny_iw^{2k-1}\rightarrow\\
 ... & ...& ... &... &... &...\\ 
 Ny_iw^{(p-1)k} & \rightarrow & Ny_iw^{(p-1)k+1} & \rightarrow &  Ny_iw^{(p-1)k+2} & \rightarrow &Ny_iw^{kp -1} \\
  \end{array}$
  In such a loop, each fiber contributes exactly  $\frac{o_N(w)}{k}=p$ elements.  Since there are $\frac{m}{d}$ elements in each fiber, $\ell$ induces $\frac{\frac{m}{d}}{ \frac{o_N(w)}{k}}$
   loops labelled $w^{o_N(w)}$ between the $o(w,j)$ fibers in  $p^{-1}(V_{w,j})$.
 \end{proof}

       \begin{rem}
      As Example \ref{ex_combined_graph} illustrate it, given $N \lhd F_n$ of index $m$  and $w \in F_n$, the number of loops and the length of the loops  in \emph{Top} of any combined graph   $\tilde{X}_{N,H}(w)$ are independent of the subgroup $H$. Indeed, $o_N(w)$ is determined by $\tilde{X}_N$ and the loops correspond to the action of $w$ on the cosets of $N$, or more generally if $w=a,b$ they describe  the action of $F_n$ on the cosets of $N$.  Yet, changing $H$ changes the fibers, the size of the fibers ($=\frac{m}{d_H}$) and the number of elements  ($=\frac{o_N(w)}{o_H(w)}$) contributed by each fiber in a loop of length $o_N(w)$.
       \end{rem}
 \section{The  HS conjecture and covering spaces}\label{sec_hs-covering}

Let $\{H_i\alpha_i\}_{i=1}^{i=s}$ be a coset  partition of $F_n$ with $H_i<F_n$ of index $d_i>1$, $\alpha_i \in F_n$, $1 \leq i \leq s$.
   Let $N_i=\bigcap\limits_{g \in F_n} g^{-1}H_ig$ of index $m_i$. Let $N=\bigcap\limits_{i=1}^{i=s}
     N_i$ of index $m$. For every $1 \leq i \leq s$, and for any $w \in F_n$, there exists a combined graph $\tilde{X}_{N,H_i}(w)$, since $N \lhd F_n$ and $N \subset H_i$. At \emph{Top} of  every combined graph $\tilde{X}_{N,H_i}(w)$, there are  $\frac{m}{o_N(w)}$ loops of length $o_N(w)$ labelled $w^{o_N(w)}$. Yet, for each $i$, \emph{Top} of  $\tilde{X}_{N,H_i}(w)$ looks different, since  the fibers and  the connections between them are different and they depend on the subgroup $H_i$ (see Section \ref{sec_combined}).  Our aim in the following section is to describe graphically the coset partition in terms of the combined graphs  $\tilde{X}_{N,H_i}(w)$, for any $w \in F_n$. In particular, we introduce a new graph, the HS-colored graph, to describe the coset  partition.
   
 \subsection{A graphical description of a coset partition of $F_n$ with the Schreier graphs}\label{sec_hs-graph}
  Let   $w \in F_n$, with order $o_N(w)$. Let  $\tilde{X}_{N}(w)$ be the Schreier graph of $N$. Let $\tilde{X}_{N,i}(w)$ denote the  $s$ combined graphs $\tilde{X}_{N,H_i}(w)$,  $1 \leq i\leq s$.  Let denote by $p_w^{-1}(H_i\alpha_i)$  the fiber over $H_i\alpha_i$ in \emph{Top} of the combined graph $\tilde{X}_{N,i}(w)$.  In each $\tilde{X}_{N,i}(w)$, $1 \leq i \leq s$,   let color in color $c_i$  the vertex representing the coset $H_i\alpha_i$ at \emph{Bot}, and all the elements in the fiber $p_w^{-1}(H_i\alpha_i)$ in \emph{Top}. We keep record of the  color attributed to each coset of $N$, and color accordingly the vertices in $\tilde{X}_{N}(w)$.

  \begin{defn}\label{defn_coloredXN}
   Let $\{H_i\alpha_i\}_{i=1}^{i=s}$ be a coset  partition of $F_n$ with $H_i<F_n$ of index $d_i$, $\alpha_i \in F_n$, $1 \leq i \leq s$. Let $N=\bigcap\limits_{i=1}^{i=s}
   \bigcap\limits_{g \in F_n} g^{-1}H_ig$ of index $m$, with Schreier graph $\tilde{X}_{N}(w)$,    $w\in F_n$. 
  \begin{enumerate}[(i)]
  \item  We define  $\bar{X}_{N}(w)$, \emph{the HS-colored  graph}, to be   $\tilde{X}_{N}(w)$  with colored vertices: a vertex is colored in color $c_i$ if and only if it  belongs to $p^{-1}_w(H_i\alpha_i)$ in \emph{Top} of $\tilde{X}_{N,i}(w)$, $1\leq i \leq s$.
  \item  We define \emph{the order of $w$  relative to  $H_i\alpha_i$},  $o_{*i}(w)$,  to be the order of $w$ in  \emph{Bot} of $\tilde{X}_{N,i}(w)$ relative to the  vertex $H_i\alpha_i$, that is the minimal natural number such that  $w^{o_{*i}(w)} \in \alpha_i^{-1}H_i\alpha_i$, or equivalently  $w^{o_{*i}(w)}$ is a loop at the vertex $H_i\alpha_i$ in  \emph{Bot} of $\tilde{X}_{N,i}(w)$. \end{enumerate}
  \end{defn}
   Since $\{H_i\alpha_i\}_{i=1}^{i=s}$ is  a coset  partition of $F_n$ and each fiber $p^{-1}_w(H_i\alpha_i)$ contains exactly $\frac{m}{d_i}$ cosets of $N$, we have that each vertex  in $\bar{X}_{N}(w)$  is colored in one and only one color, and
     $m=\sum\limits_{i=1}^{i=s}\frac{m}{d_i}$. 
  
  \begin{ex}\label{ex_partitionF2-4-cycle}
  Consider this partition of $F_2$: $F_2= H_1 \bigcup H_2a \bigcup H_3ab$, where $H_1=H=\langle b,a^2,aba\rangle$ and $H_2=H_3=K$ from Example \ref{ex_cov_subgrps}. Let $N \lhd F_2$ of index $8$ from Example \ref{ex_combined_graph}, $N$ is the intersection of the normal cores of $H_1$, $H_2$ and  $H_3$. 
    Let $w=ab$ with $o_N(w)=4$, $o_{*1}(w)=2$, $o_{*2}(w)=o_{*3}(w)=4$.  The HS-colored graph $\bar{X}_{N}(w)$ is described below:
       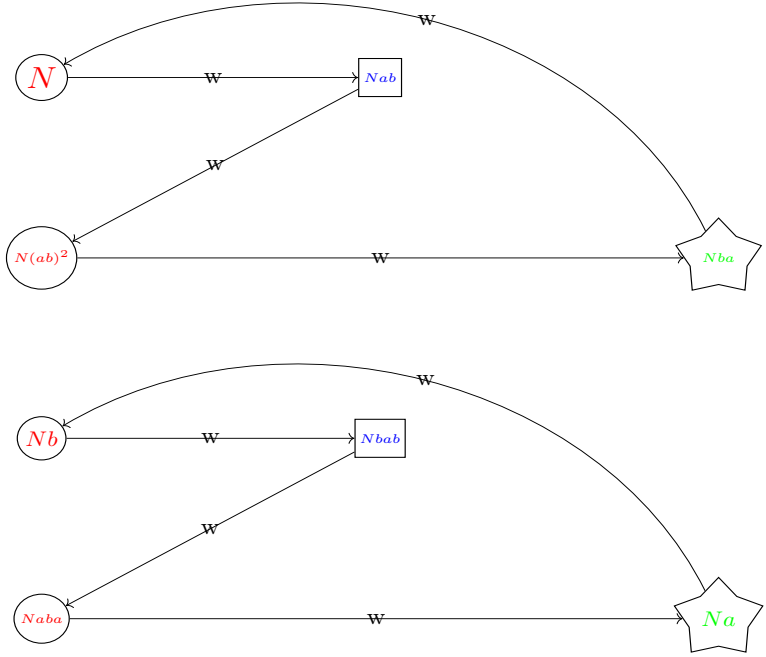
\begin{figure}[H]
\centering   \scalebox{0.9}[0.8]{\begin{tikzpicture}
         \SetGraphUnit{10}
        \tikzset{VertexStyle/.append  style={fill}}
      
      \Vertex[L=\color{red}{$N$}, x=-5, y=3]{B}
          \Vertex[L=\tiny\color{red}{$N(ab)^2$}, x=-5,y=0]{A}
           \Vertex[L=\small\color{red}{$Nb$}, x=-5,y=-3]{X}
            \Vertex[L=\tiny\color{red}{$Naba$}, x=-5,y=-6]{Y}
         \SetGraphUnit{4}
          \tikzset{VertexStyle/.append  style={fill}}
           \tikzset{every loop/.style={min distance=10mm,looseness=10}}


    \tikzset{VertexStyle/.append style={rectangle}}  
     \Vertex[L=\tiny\color{blue}{$Nab$}, x=0, y=3]{D}
      \Vertex[L=\tiny\color{blue}{$Nbab$}, x=0, y=-3]{Z}
       \tikzset{VertexStyle/.append style={star}} 
     \Vertex[L=\tiny\color{green}{$Nba$}, x=5,y=0]{C}
      \Vertex[L=\small\color{green}{$Na$}, x=5,y=-6]{W}     
         \path[->]          (B)  edge                 node {w} (D);
           \path[->]          (D)  edge      node {w} (A);
          \path[->]          (A)  edge     node {w} (C);
           \path[->]          (C)  edge        [bend right=50]          node [swap] {w} (B);
             
                   \path[->]          (X)  edge                 node {w} (Z);
                     \path[->]          (Z)  edge      node {w} (Y);
                    \path[->]          (Y)  edge     node {w} (W);
                     \path[->]          (W)  edge        [bend right=50]          node [swap] {w} (X);
             \end{tikzpicture}}\\
             \caption{The  HS-colored graph $\bar{X}_{N}(w)$, with $N\lhd F_2$, $w=ab$.}\label{fig_ex_partition_F2}
   \end{figure}  \end{ex} 
The vertices in  $p^{-1}_w(H_1)=\{N,N(ab)^2,Nb,Naba\}$ are colored in red (circles), those in $p^{-1}_w(H_2a)=\{Nba,Na\}$  in green (stars) and those in $p^{-1}_w(H_3ab)=\{Nab,Nbab\}$  in blue (squares). 
From Lemma \ref{lem_number_loops}, the fiber $p^{-1}_w(H_1)$ participates in $\frac{\frac{m}{d_1}}{\frac{o_N(w)}{o_{*1}(w)}}=2$ loops and in each loop it contributes $\frac{o_N(w)}{o_{*1}(w)}=2$ elements, the fiber $p^{-1}_w(H_2a)$ participates in $\frac{\frac{m}{d_2}}{\frac{o_N(w)}{o_{*2}(w)}}=2$ loops and in each loop it contributes $\frac{o_N(w)}{o_{*2}(w)}=1$ element and the same holds for the fiber $p^{-1}_w(H_3ab)$.   \\
   
  Before we go on with $F_n$ with $n\geq 2$, we give a  graphical description of  a coset partition in the special case of  $F_1=\mathbb{Z}$.
\subsection{A graphical description of  a coset partition of $\mathbb{Z}$}
 The group of integers $\mathbb{Z}=\langle 1 \rangle$ is the fundamental group of the bouquet $X$ with only one circle labelled $1$.
 For each subgroup $H=o\mathbb{Z}$ of index $o>1$, there exists a $o$-fold covering $\tilde{X}_H$, which is  an oriented loop of length $o>1$ and labelled $o$.  Let $\{o_i\mathbb{Z} +r_i\}_{i=1}^{i=t}$, $r_i \in \mathbb{Z}$, be  a  coset partition of $\mathbb{Z}$. Let fix $w=1$, so that  $o_i(1)=o_i$, and $\tilde{X}_{H}$ and $\tilde{X}_{H}(1)$ coincide, for any $H<\mathbb{Z}$. Let $N=o_N\mathbb{Z}$, with $o_N$  divisible by $\ell cm(o_1,...,o_t)$, so $N \subseteq \bigcap\limits_{i=1}^{i=t}o_i\mathbb{Z}$.  Let  $\bar{X}_{N}$ be the HS-colored graph as defined in Section \ref{sec_hs-graph}. Since $o_N(1)=o_N$, there is a unique loop of length $o_N$. So,  all the $t$ fibers participate in this unique loop and from Lemma \ref{lem_number_loops}, each fiber contributes all its $\frac{o_N}{o_i}$$(=\frac{o_N(1)}{o_i(1)})$ elements in this  loop. That is,  $\sum_{i=1}^{i=t}\frac{o_N}{o_i}=o_N$. And conversely, 
  \begin{lem}\label{lem_formule_sum=on}
  Let $c_i$, $1 \leq i \leq t$, be different colors. Let $o_N$, $o_i>1$, $1 \leq i \leq t$, be natural numbers such that  $\ell cm(o_1,...,o_t)$ divides  $o_N$. 
 If there exists a loop $\ell$ of length $o_N =\sum_{i=1}^{i=t}\frac{o_N}{o_i}$, with each package of $\frac{o_N}{o_i}$ vertices colored in $c_i$, $1 \leq i \leq t$ and such that any two consecutive vertices colored in the same color $c_i$ are at $o_i$  distance.
 Then there exist $r_i \in\mathbb{Z}$, $1 \leq i \leq t$, such that  $\{o_i\mathbb{Z} +r_i\}_{i=1}^{i=t}$ is a  coset partition of $\mathbb{Z}$.
  \end{lem}  
 \begin{proof}
   We set $N=o_N\mathbb{Z}$ and $H_i=o_i\mathbb{Z}$, $1 \leq i \leq t$. Since  $\ell cm(o_1,...,o_t)$ divides  $o_N$, $N \subseteq \bigcap\limits_{i=1}^{i=t}H_i$.  We show that a  loop $\ell$ satisfying  the conditions of the lemma describes the Schreier  graph of $N$  and  a coset partition   $\{o_i\mathbb{Z} +r_i\}_{i=1}^{i=t}$. Choose any vertex in $\ell$ and call it $N$ and the adjacent vertex $N+1$ and so on until $N+o_N-1$. If in $\ell$, any two consecutive vertices colored in the same color $c_i$ are at $o_i$  distance, then there exist $r_i \in\mathbb{Z}$, $1 \leq i \leq t$, such that the vertices colored $c_i$ are $N+r_i$, $N+r_i +o_i$, $N+r_i +2o_i$,...,$N+r_i +(\frac{o_N}{o_i}-1)o_i$,  that is these   are  $\frac{o_N}{o_i}$ cosets of $N$ whose union is $o_i\mathbb{Z}+r_i$. Next, since each vertex in  $\ell$ is  colored in only one color, the sets $o_i\mathbb{Z}+r_i$, $1 \leq i \leq t$, are all disjoint. At last,  $o_N =\sum_{i=1}^{i=t}\frac{o_N}{o_i}$ implies that $\mathbb{Z}$ is partitioned  by these disjoint cosets. 
 \end{proof} 
 
 \begin{thm}\cite[Theorem 2]{newman}\cite[Theorems 4,5]{znam}\label{theo_Z}\\
  Let  $\{o_i\mathbb{Z} +r_i\}_{i=1}^{i=t}$  be a coset partition of $\mathbb{Z}$.  Then: 
    \begin{enumerate}[(i)]
    \item The natural numbers $\{o_{i}\mid 1 \leq i \leq t\}$ are not pairwise prime.
      \item  $o_{max}=Max \{o_{i}\mid  1 \leq i \leq t\}$ appears at least $p$ times, where $p$ is the smallest prime dividing $o_{max}$.
      \item Any $o_{l}$ divides another $o_{k}$, where $ 1 \leq k \neq l \leq t$.
      \item Any $o_{k}$ that does not properly divide any other $o_{i}$  appears at least twice, $ 1 \leq k,i\leq t$.
        \end{enumerate} 
  \end{thm}
 \section{Main results on the  Herzog-Sch\"onheim conjecture for  free groups of finite rank}
  In this section, we give some conditions that ensure  a partition of the free group of finite rank has multiplicity. The notation is the following for the rest of this section. Let $F_n$ be the free group on $n \geq 2$ generators. Let $\{H_i\alpha_i\}_{i=1}^{i=s}$ be a coset  partition of $F_n$ with $H_i<F_n$ of index $d_i>1$, $\alpha_i \in F_n$, $1 \leq i \leq s$, and $d_1 \leq ...\leq d_s$.  Let $N_i=\bigcap\limits_{g \in F_n} g^{-1}H_ig$ of index $m_i$. Let $N=\bigcap\limits_{i=1}^{i=s}
          N_i$ of index $m$.  Let   $w \in F_n$.  Let $\tilde{X}_{i}(w)$ denote the  Schreier  graph $\tilde{X}_{H_i}(w)$,  $1 \leq i\leq s$ and  $\tilde{X}_{N}(w)$ be the Schreier graph of $N$. Let $\tilde{X}_{N,i}(w)$ denote the   combined graph $\tilde{X}_{N,H_i}(w)$,  $1 \leq i\leq s$.
           Let denote by $p_w^{-1}(H_i\alpha_i)$  the fiber over $H_i\alpha_i$ in \emph{Top} of the combined graph $\tilde{X}_{N,i}(w)$.  Let $\bar{X}_{N}(w)$ be the HS-colored  graph.
            Let  $o_{*i}(w)$,  $1 \leq i \leq s$,  be the order of $w$ relative to the  vertex $H_i\alpha_i$.   
 \subsection{Description of the loops in the HS-colored graph}
 
   In \emph{Top} of any   combined graph, there are the same  $\frac{m}{o_N(w)}$ disjoint loops of length $o_N(w)$ from   $\tilde{X}_{N}(w)$. But, for different subgroups $H_i$ and $H_j$, the fibers in \emph{Top} of $\tilde{X}_{N,i}(w)$ and $\tilde{X}_{N,j}(w)$ are different and so  \emph{Top} of $\tilde{X}_{N,i}(w)$ and $\tilde{X}_{N,j}(w)$ look also  different. The HS-colored  graph $\bar{X}_{N}(w)$ describes the  partition  $\{H_i\alpha_i\}_{i=1}^{i=s}$,   with each fiber  $p^{-1}_w(H_i\alpha_i)$ colored in a different color   (see Figures \ref{fig-combined N-K},  \ref{fig_graph_combineN_H} and \ref{fig_ex_partition_F2}). The question now is how the colored fibers are connected in $\bar{X}_{N}(w)$. In the following lemmas, we aim to answer this question and describe the loops in $\bar{X}_{N}(w)$.\\
 Note that $w \in \bigcap\limits_{i=1}^{i=s}\alpha_i^{-1}H_i\alpha_i$ if and only if $o_{*i}(w)=1$ for every $1 \leq i \leq s$, so if $w \in F_n \setminus \bigcap\limits_{i=1}^{i=s}\alpha_i^{-1}H_i\alpha_i$, then there exists $1 \leq i_1 \leq s$ such that  $o_{*i_1}(w)>1$.
  \begin{lem}\label{lem_loops_s_fibers}
Let $w \in F_n \setminus \bigcap\limits_{i=1}^{i=s}\alpha_i^{-1}H_i\alpha_i$, with order $o_N(w)$. Then in  $\bar{X}_{N}(w)$,  the HS-colored graph, the following occurs:
   \begin{enumerate}[(i)]
   \item There are $\frac{m}{o_N(w)}$ disjoint loops of length
    $o_N(w)$ (labelled $w^{o_N(w)}$).
   \item If the fiber $p_w^{-1}(H_i\alpha_i)$ participates in a loop, then it contributes $\frac{o_N(w)}{o_{*i}(w)}$ elements and the distance  between two consecutive vertices  is $o_{*i}(w)$.
   \item Each fiber $p_w^{-1}(H_i\alpha_i)$ participates in $\frac{\frac{m}{d_i}}{\frac{o_N(w)}{o_{*i}(w)}}$ loops.
   \item A fiber $p_w^{-1}(H_i\alpha_i)$ participates in all the  $\frac{m}{o_N(w)}$ loops if and only if $o_{*i}(w)=d_i$.
   \item Two fibers $p_w^{-1}(H_i\alpha_i)$ and $p_w^{-1}(H_j\alpha_j)$ participate in  the same number of loops    if and only if $\frac{o_{*i}(w)}{d_i}=\frac{o_{*j}(w)}{d_j}$.
   \item For each $j$-th loop, $1 \leq j \leq \frac{m}{o_N(w)}$,  the following equation holds:
   \begin{equation}\label{eqn}
   \sum\limits_{i \in Y_j} \frac{o_{N}(w)}{o_{*i}(w)}=o_N(w)
   \end{equation}
  where $Y_j=\{1 \leq i \leq s \mid p^{-1}(H_i\alpha_i)$ participates in the $j$-th loop $\}$.
     \end{enumerate} 
  \end{lem}
  
 \begin{proof}
 $(i),(ii),(iii)$ result from Lemma \ref{lem_number_loops}, since  from the definition of the HS-colored graph, $\bar{X}_{N}(w)$ is just $\tilde{X}_{N}(w)$  with colored vertices.\\
 $(iv)$ results from $(iii)$: $\frac{\frac{m}{d_i}}{\frac{o_N(w)}{o_{*i}(w)}}= \frac{m}{o_N(w)}$
 if and only if $o_{*i}(w)=d_i$.\\
 $(v)$ results from $(iii)$: $\frac{\frac{m}{d_i}}{\frac{o_N(w)}{o_{*i}(w)}}= \frac{\frac{m}{d_j}}{\frac{o_N(w)}{o_{*j}(w)}}$ 
 if and only if $\frac{o_{*i}(w)}{d_i}=\frac{o_{*j}(w)}{d_j}$.\\
 $(vi)$ Consider a  $j$-th loop, $1 \leq j \leq \frac{m}{o_N(w)}$. If the fiber $p^{-1}(H_{i_0}\alpha_{i_0})$ participates in this loop, then from $(iii)$ it contributes 
 $\frac{o_N(w)}{o_{*{i_0}}(w)}$ elements. If $\frac{o_N(w)}{o_{*{i_0}}(w)}=o_N(w)$, that is  $o_{*{i_0}}(w)=1$, then Equation \ref{eqn} holds trivially for this loop. Since $w  \notin \bigcap\limits_{i=1}^{i=s}\alpha_i^{-1}H_i\alpha_i$, there is some $1 \leq i_1 \leq s$ such that $o_{*{i_1}}(w)>1$ and there are loops with more than one participating fiber. Assume the  fiber $p_w^{-1}(H_{i_1}\alpha_{i_1})$ participates in such a  loop, then from $(iii)$ it contributes 
 $\frac{o_N(w)}{o_{*{i_1}}(w)}$ elements colored in $c_{i_1}$. Next,  another  fiber,  $p_w^{-1}(H_{i_2}\alpha_{i_2})$, participates and contributes 
 $\frac{o_N(w)}{o_{*{i_2}}(w)}$ elements colored in $c_{i_2}$ (completely disjoint with the former). If $\frac{o_N(w)}{o_{*{i_1}}(w)}+\frac{o_N(w)}{o_{*{i_2}}(w)} =o_N(w)$,  Equation \ref{eqn} holds for this loop. Otherwise, if  $\frac{o_N(w)}{o_{*{i_1}}(w)}+\frac{o_N(w)}{o_{*{i_2}}(w)} <o_N(w)$, then  there is a fiber,  $p_w^{-1}(H_{i_3}\alpha_{i_3})$, that participates and contributes 
 $\frac{o_N(w)}{o_{*{i_3}}(w)}$ elements colored in $c_{i_3}$. We repeat the process until all the $o_N(w)$ vertices in this loop are counted and Equation \ref{eqn} holds  for this loop. Since the  $\frac{m}{o_N(w)}$ loops are disjoint, we do the process for each loop independently.
  \end{proof} 
  \begin{prop}\label{prop_repeat_omax}
  Let $w \in F_n \setminus \bigcap\limits_{i=1}^{i=s}\alpha_i^{-1}H_i\alpha_i$, with order $o_N(w)$.   For each $j$-th loop, $1 \leq j \leq \frac{m}{o_N(w)}$, let   $Y_j=\{1 \leq i \leq s \mid p^{-1}(H_i\alpha_i)$ participates in the $j$-th loop$\}$. Then, in  $\bar{X}_{N}(w)$, the following occurs: 
    \begin{enumerate}[(i)]
    \item The natural numbers $\{o_{*i}(w)\mid i \in Y_j\}$ are not pairwise prime.
      \item  $o_{max,j}=Max \{o_{*i}(w)\mid i \in Y_j\}$ appears at least $p$ times, where $p$ is the smallest prime dividing $o_{max,j}$.
      \item Any $o_{*l}(w)$ divides another $o_{*k}(w)$, for $k,l \in Y_j$.
      \item Any $o_{*k}(w)$ that does not properly divide any other $o_{*i}(w)$ appears at least twice, for  $i,k \in Y_j$.
        \end{enumerate} 
  \end{prop}
  \begin{proof}
  We consider only loops with more than one participating fiber.
   From Lemma \ref{lem_loops_s_fibers}, any  $j$-th loop is of length $o_N(w)$ and satisfies the following three conditions:  $o_N(w)=\sum\limits_{i \in Y_j} \frac{o_{N}(w)}{o_{*i}(w)}$, each package of $\frac{o_{N}(w)}{o_{*i}(w)}$ vertices are colored in color $c_i$, $i \in Y_j$, and  any two consecutive vertices colored in the same color $c_i$ are at $o_{*i}(w)$  distance.
   So, from Lemma \ref{lem_formule_sum=on}, each such loop describes a coset partition of  $\mathbb{Z}$,
$\{o_{*i}(w)\mathbb{Z} +r_i\}_{i\in Y_j}$
$r_i \in\mathbb{Z}$.\\
   $(i),(ii),(iii),(iv)$  result then from the HS in $\mathbb{Z}$ and are a transcription of  $(i),(ii),(iii),(iv)$  from Theorem \ref{theo_Z} for a $j$-th loop. 
  \end{proof}
 
   \subsection{Main Results on the  Herzog-Sch\"onheim conjecture for free groups of finite rank} \label{sec-first-results}
  We prove the first Theorems  that provide sufficient conditions for multiplicity  in  the  coset partition of $F_n$. 
   \begin{thm}\label{theorem_HS_condition_w-exists}
Let $F_n$ be the free group on $n \geq 2$ generators. Let $\{H_i\alpha_i\}_{i=1}^{i=s}$ be a coset  partition of $F_n$ with $H_i<F_n$ of index $d_i$, $\alpha_i \in F_n$, $1 \leq i \leq s$, and $1 <d_1 \leq ...\leq d_s$.  Let $\tilde{X}_{i}$ denote the   Schreier  graph  of $H_i$,  $1 \leq i \leq s$. For $g\in F_n$,  let $o_{*i}(g)$ denote the order of $g$ relative to  $H_i\alpha_i$, $1 \leq i\leq s$.  If there exists $w \in F_n$ such that  $o_{*s}(w)=d_s$, then the index $d_s$ appears in the partition at least  $p$ times, where $p$ is the smallest prime dividing $d_s$. 
 \end{thm}
 \begin{proof}
 If there exists  $w \in F_n$  such that  $o_{*s}(w)=d_s$, then   $o_{max}(w)=Max\{o_{*i}(w) \mid 1 \leq i\leq s \}\;=d_{s}$.  Furthermore, from Lemma \ref{lem_loops_s_fibers}$(iv)$,  the fiber $p_w^{-1}(H_s\alpha_s)$ participates in all the  loops. Consider any $j$-th  loop. From Proposition \ref{prop_repeat_omax},  $o_{max}(w)$  appears at least $p$ times in this loop, where $p$ is the smallest prime dividing $o_{max}(w)$. That is, there are at least  $1 \leq i_1,..,i_{p-1} \leq s-1$ such that  $o_{*i_1}(w)=...=o_{*i_{p-1}}(w)=o_{max}=d_s$. But,  $o_{*i}(w)\leq d_i$, so $d_{i_1}=...=d_{i_{p-1}}=d_s$.
\end{proof}
Note that if  there exists $w \in F_n$, such that $o_{*i}(w)=d_i$  for every $1 \leq i\leq s$, then multiplicities occur as in $\mathbb{Z}$. In particular, it ensures the existence of a partition of $\mathbb{Z}$ with subgroups of indices $d_1,...,d_s$. For the partition of $F_2$ in Example \ref{ex_partitionF2-4-cycle}, $o_{*1}(ab)=d_1=2$, $o_{*2}(ab)=d_2=o_{*3}(ab)=d_3=4$. If $H_s \lhd F_n$, then there exists a $d_s$-cycle in $T_s$ if and only if $T_s\simeq \mathbb{Z}_{d_s}$.
\begin{ex}
 Consider the partition of $F_2$ given in Example \ref{ex_partitionF2-4-cycle}.  The  subgroup  of maximal index is $K$ with index  $4$ (see Figure \ref{fig-4-shheted-covering-K}). For $w=ab$, $o_K(ab)=4$  in   $\tilde{X}_{K}$ (see Figure \ref{fig-combined N-K}),  so  Theorem \ref{theorem_HS_condition_w-exists}$(i)$ ensures the HS is satisfied and in this example  the index $4$ appears exactly twice. 
 \end{ex}

We keep the same notation as in Theorem \ref{theorem_HS_condition_w-exists}.  For $g\in F_n$,  let $o_{max}(g)=Max\{o_{*i}(g) \mid 1 \leq i\leq s \}$, and   $p$  be the smallest prime dividing $o_{max}(g)$. We define  $\#=\mid \{1 \leq i\leq s \mid o_{*i}(g)=o_{max}(g)\}\mid$. 
\begin{thm}\label{theorem_HS_more-condition_w-exists}
Let $F_n$ be the free group on $n \geq 2$ generators. Let $\{H_i\alpha_i\}_{i=1}^{i=s}$ be a coset  partition of $F_n$ with $H_i<F_n$ of index $d_i$, $\alpha_i \in F_n$, $1 \leq i \leq s$, and $1<d_1 \leq ...\leq d_s$.  
 Let $r \in \mathbb{Z}$,  $4 \leq r \leq s-1$. If there exists $w \in F_n$ that satisfies one of the following conditions:
 \begin{enumerate}[(i)]
\item  $o_{max}(w)>d_{s-2}$.
 \item  $o_{max}(w)>d_{s-3}$, $p\geq 3$.
 \item  $o_{max}(w)>d_{s-3}$, $p=2$, and  $\# \geq 4$.
 \item  $o_{max}(w)>d_{s-3}$, $p=2$, and $\#=2$ 
 \item  $o_{max}(w)>d_{s-r}$,  $p\geq r$.
    \item  $o_{max}(w)>d_{s-r}$,   $\# \geq r+1$. 
  \item    $o_{max}(w)>d_{s-r}$,  $\#=p$. 
 \end{enumerate} 
 Then  the coset partition $\{H_i\alpha_i\}_{i=1}^{i=s}$ has multiplicity.
 \end{thm}
 \begin{proof}
   $(i)$  Let  $k_w$ denote $o_{max}(w)$.  If there exists  $w \in F_n$   with $k_w>d_{s-2}$, then either $o_{*s}(w)=k_w$  or $o_{*s-1}(w)=k_w$. From Proposition \ref{prop_repeat_omax},  $o_{max}(w)$  appears at least twice, so  $o_{*s}(w)=o_{*s-1}(w)=k_w$.  If $k_w$ occurs exactly twice, then the fibers $p_w^{-1}(H_s\alpha_s)$ and  $p_w^{-1}(H_{s-1}\alpha_{s-1})$ participate together in the same loops (otherwise there would be a repetition of $k_w$ for an additional fiber). So, from Lemma \ref{lem_loops_s_fibers}$(v)$,  
   $\frac{o_{*s}(w)}{d_s}=\frac{o_{*s-1}(w)}{d_{s-1}}$ implies $d_{s-1}=d_{s}$. 
     If $k_w$ occurs  three times or more, then  there is  at least one additional fiber $p_w^{-1}(H_i\alpha_i)$, with  $o_{*i}(w)=k_w>d_{s-2}$ and $o_{*i}(w)\leq d_i \leq d_{s-2}$. So, the pigeonhole principle can be  applied: there are two values of indices $d_s$ and $d_{s-1}$ to distribute to at least three fibers, that is $d_i=d_s$ or $d_i=d_{s-1}$, and  there is a repetition of one of the indices. \\  
 $(ii)$,  $(iii)$,  $(v)$, $(vi)$ The proof relies on the pigeonhole principle as it was applied in the proof of $(i)$.\\
$(iv)$,  $(vii)$ Whenever $\#=p$, there are $p$ fibers that participate together in the same loops. So,  Lemma \ref{lem_loops_s_fibers}$(v)$ implies they have the same index. 
 \end{proof}
Theorem \ref{theo0}  and Theorem \ref{theo1} from the introduction  are  computationally speaking far easier to check than Theorems \ref{theorem_HS_condition_w-exists} and \ref{theorem_HS_more-condition_w-exists}. We show  that in fact  Theorems \ref{theo0}  and \ref{theo1} are  equivalent to  Theorems \ref{theorem_HS_condition_w-exists} and \ref{theorem_HS_more-condition_w-exists} respectively. For that, we need  the following lemma. In the following, we assume the  permutations  are decomposed into the  product of disjoint cycles.
 \begin{lem}\label{lem_o(w,i)=d}
 Let $H < F_n$ be of finite index $d$. Let $\tilde{X}_H$ be the  Schreier graph of $H$ with vertices $\tilde{x}_0,...,\tilde{x}_{d-1}$. Let $T \leq S_d$ be  the transition group of $\tilde{X}_H$. Then
 \begin{enumerate}[(i)]
\item There  exists a $d$-cycle in $T$, if and only if there exists  $w \in F_n$ such that   $o(w,j)=d$, for some $0 \leq j  \leq d-1$.
\item If there  exists a permutation in $T$ with a   $k$-cycle at some   $\tilde{x}_{j}$, then for   every $\tilde{x}_{r}$, there exists  a permutation  in $T$ with a   $k$-cycle at $\tilde{x}_{r}$.
 \end{enumerate}
 \end{lem}
 \begin{proof}
 $(i)$ Let $w \in F_n$. Assume    $o(w,j)=d$, for some $0 \leq j  \leq d-1$. Then, from Lemma \ref{lem_o(w,i)}$(i)$,   $o(w,j)=d$, for all  $0 \leq j  \leq d-1$, that is  $w^d \in N_H$, where $N_H$ is the normal core of $H$. So, $o(N_Hw)=d$ in $^{F_n}\big/_{N_H}  \simeq T$ with $T \leq S_d$ and $w$ acts as a $d$-cycle on the right cosets of $H$. Conversely, assume  there is a $d$-cycle  $\phi$  in $T$. Let $w\in F_n$ be any inverse image of $\phi$. Then    $o(w,j)=d$,  for all  $0 \leq j  \leq d-1$.\\
 $(ii)$ Assume first that  $\sigma \in T$ is a permutation  with a   $k$-cycle at  $\tilde{x}_{0}$.  Let $g\in F_n$ be any inverse image of $\sigma$. Then there is a loop labelled $g^k$ at $\tilde{x}_{0}$. Let $w=t_r^{-1}gt_r$,  with $t_r$  the label of a minimal path from $\tilde{x}_0$ to $\tilde{x}_r$. Then there is a loop labelled $w^k$ at $\tilde{x}_{r}$ and in the   corresponding permutation in $T$ there is  a  $k$-cycle. The same proof applies for any  $\tilde{x}_{j}$. 
 \end{proof}
\begin{proof}[Proofs of Theorem \ref{theo0}  and Theorem \ref{theo1} ]
Using  Lemma \ref{lem_o(w,i)=d}$(i)$,  Theorem \ref{theo0}  and Theorem \ref{theo1} from the introduction are  the transcription in terms of cycles of Theorems \ref{theorem_HS_condition_w-exists} and \ref{theorem_HS_more-condition_w-exists} respectively.
\end{proof}
\section{The space of coset partitions of $F_n$}\label{sec_space_partitions}
Let $F_n$ be the free group on $n \geq 2$ generators.  We define $Y'$ to be  the space of coset partitions of $F_n$ (only with subgroups of finite index). For each subgroup $H$ of $F_n$ of  finite index $d>1$, there exists a partition of $F_n$ by the $d$ cosets of $H$.  Generally, if  $P \in Y'$, then $P=\{H_i\alpha_i\}_{i=1}^{i=s}$,  a coset  partition of $F_n$ with $H_i<F_n$ of index $d_i$, $\alpha_i \in F_n$, $1 \leq i \leq s$, and $1<d_1 \leq ...\leq d_s$.  To get  some intuition on  $Y'$, it is worth recalling that  the subgroup growth of $F_n$ is exponential.
\subsection{Action of $F_n$ on the space of its coset partitions}\label{subsec_action_F_n-on-space_partitions}
  There exists a  natural right action of $F_n$ on $Y'$. Indeed, if $w \in F_n$, then $P \cdot w = P'$, with $P'= \{H_i\alpha_i\,w\}_{i=1}^{i=s}$ in $Y'$. 
\begin{lem}
The natural right  action of $F_n$ on $Y'$ is faithful 
\end{lem}
\begin{proof}
 Let $w \in F_n$. Then  $P \cdot w = P$ for every $P \in Y'$  if and only if  $w$ belongs to the intersection of all the subgroups of finite index of $F_n$. As $F_n$ is residually finite \cite[p.158]{robinson},  the intersection of all the subgroups of finite index of $F_n$ is trivial, so $w=1$, that is the action is faithful.
\end{proof}
Let $P=\{H_i\alpha_i\}_{i=1}^{i=s}$ in $Y'$ and let $\tilde{X}_i$ be the Schreier graph of $H_i$, $ 1 \leq i \leq s$. Let $w \in F_n$. We denote by $o_{*i}(w)$ the minimal natural number, $1 \leq o_{*i}(w) \leq d_i$, such that $w^{o_{*i}(w)}$  is  a loop  at the vertex $H_i\alpha_i$ in $\tilde{X}_{i}$ or equivalently $w^{o_{*i}(w)} \in \alpha_i^{-1}H_i\alpha_i$ (see Section 4.1).
\begin{lem}
Let  $P=\{H_i\alpha_i\}_{i=1}^{i=s}$ in $Y'$. Then $\mid Orb_{F_n}(P) \mid \leq d_1...d_s$, where $ Orb_{F_n}(P)$ denotes the orbit of $P$ under the action of $F_n$.
Furthermore, for $w \in F_n$,  $\mid Orb_{w}(P) \mid =lcm(o_{*1}(w),...,o_{*s}(w))$, where $ Orb_{w}(P)$ denotes the orbit of $P$ under the action of $\langle w \rangle$.
\end{lem}
\begin{proof}
From the definition of the action of $F_n$ on $P$, $F_n$ permutes between the cosets of $H_1$, between the cosets of $H_2$ and so on. So, $\mid Orb_{F_n}(P) \mid \leq d_1...d_s$. The size of $ Orb_{w}(P)$ is equal to the minimal natural number such that  $P \cdot w^k = P$, that is  $k=lcm(o_{*1}(w),...,o_{*s}(w))$. Indeed,     $P \cdot w^k = P$ $\Leftrightarrow$ $w^k \in \bigcap\limits^{i=s}_{i=1}\alpha_i^{-1}H_i\alpha_i$ $\Leftrightarrow$ $lcm(o_{*1}(w),...,o_{*s}(w))$ divides $k$ (Lemma 4.10). 
As $k=lcm(o_{*1}(w),...,o_{*s}(w))$ is minimal such that  $P \cdot w^k = P$,   $\mid Orb_{w}(P) \mid =lcm(o_{*1}(w),...,o_{*s}(w))$.
\end{proof}

\begin{proof}[Proof of Theorem \ref{theo2}]
$(i)$, $(ii)$ From the assumption, there exists  $w \in  \bigcap \limits_{i\neq j,k}\alpha_i^{-1}H_i\alpha_i$, $w \notin \bigcap \limits_{i=1}^{i=s}\alpha_i^{-1}H_i\alpha_i$. Then $P \cdot w=\;\{H_i\alpha_iw\}_{i=1}^{i=s}$ gives $F_n=\bigcup\limits_{i \neq j,k} H_i\alpha_i \cup H_j\alpha_jw\cup H_k\alpha_kw$.
So, $H_j\alpha_jw\cup H_k\alpha_kw=H_j\alpha_j\cup H_k\alpha_k$, with $H_j\alpha_jw \neq H_j\alpha_j$ and $H_k\alpha_kw\neq H_k\alpha_k$. As $H_j\alpha_jw \cap H_j\alpha_j =\emptyset$  and $H_k\alpha_kw\cap H_k\alpha_k =\emptyset$, $H_j\alpha_jw \subseteq H_k\alpha_k $ and $H_k\alpha_kw \subseteq H_j\alpha_j$. From $H_j\alpha_jw\cup H_k\alpha_kw=H_j\alpha_j\cup H_k\alpha_k$ again, we have $H_j\alpha_jw= H_k\alpha_k$ and $H_j\alpha_j= H_k\alpha_kw$, that is $H_k\alpha_kw\alpha_j^{-1}=H_j$ a subgroup, so $H_k\alpha_kw\alpha_j^{-1}=H_k$, that is $H_k=H_j$. Furthermore,    $o_{*j}(w)=o_{*k}(w)=2$,  $w^2 \in  \bigcap \limits_{i=1}^{i=s}\alpha_i^{-1}H_i\alpha_i$.
\end{proof} 

If  the condition in  Theorem \ref{theo2} holds, then  $m$, $m_j$, $m_k$  are necessarily even.

  \begin{ex} \label{ex_partitionF2-no4cycle}
    Consider another  partition of $F_2$, with all the subgroups normal in $F_2$: $F_2= H_1 \bigcup H_2a \bigcup H_3ab$, where $H_1=H=\langle b,a^2,aba\rangle$ and $H_2=H_3=M$ with $\tilde{X}_{M}$ described below:
    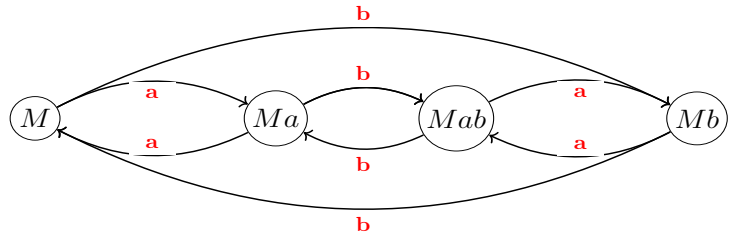
\begin{figure}[H]
     \centering\scalebox{0.8}[0.7]{\begin{tikzpicture}
     \SetGraphUnit{4}
      \tikzset{VertexStyle/.append  style={fill}}
       \Vertex[L=$M$, x=-4,y=0]{A}
       \Vertex[L=$Ma$, x=0, y=0]{B}
     
     \Vertex[L=$Mab$, x=3, y=0]{C}
      \Vertex[L=$Mb$, x=7, y=0]{D}
     \Edge[label = a, labelstyle = below](A)(B)
      \Edge[label =a, labelstyle =above](B)(A)
     \Edge[label = b, labelstyle = above](B)(C)
     \Edge[label = b, labelstyle = below](C)(B)
      \Edge[label = b, labelstyle = above](B)(C)
       \Edge[label = a, labelstyle = below](C)(D)
     \Edge[label = a, labelstyle = above](D)(C)
     \Edge[label = b, labelstyle = above](A)(D)
       \Edge[label = b, labelstyle = below](D)(A)
   \end{tikzpicture}}
   
      \caption{$\tilde{X}_{M}$ for the normal subgroup $M$ of index $4$.}\label{fig-Mnormal-no-4cycle}
       \end{figure}
  The  subgroup  of maximal index is $M$ with index  $4$ (see Figure \ref{fig-Mnormal-no-4cycle}).  The transition group $T_M$  of $\tilde{X}_{M}$ is isomorphic to $\mathbb{Z}_2\times \mathbb{Z}_2$. So, there is no $4$-cycle in $T_M$,  no cycle of length greater than $2$ ($2=d_{1}=d_{s-2}$) neither, and  Theorem \ref{theo0} cannot apply.  Nevertheless, Theorem \ref{theo2} applies. Indeed, the subgroup $\bigcap\limits_{i=1}^{i=3}H_i\subsetneq H_1$, and this implies $H_2=H_3$.
   \end{ex}  
We recall that for each subgroup $H$ of index $d$ in $F_n$ (or in any group), there is a transitive action of the group on the set of right cosets of $H$, that is given two cosets $H\alpha$ and  $H\beta$ of $H$, there exists $w$ such that $H\alpha\cdot w=\,H\beta$. So, the following question arises:
\begin{ques}\label{ques}
Let $P=\{H_i\alpha_i\}_{i=1}^{i=s}$ and  $P'=\{H_i\beta_i\}_{i=1}^{i=s}$ in $Y'$. Does  there necessarily exist $w \in F_n$ such that $P'=P\cdot w$ ?
\end{ques}

\subsection{Topology in the space of coset partitions of $F_n$}
We refer to \cite{munkres} for more details.
Let  $Y'$  be  the space of coset partitions of $F_n$. Given $P=\{H_i\alpha_i\}_{i=1}^{i=s}$ in $Y'$, with $d_s \geq ...\geq d_1>1$, we identify $P$ with the $s$-tuple $(H_s,...,H_1)$ and we consider $H_s$ at the first place, $H_{s-1}$ at the second place and so on. Let  $P'\in Y'$, $P'= \{K_i\beta_i\}_{i=1}^{i=t}$. 
We define a function $d: Y' \times Y' \rightarrow\mathbb{R}$ :
\[ d(P,P')= \left\{\begin{array}{ccc}
2^{-k} & & \text{if } k\, \text{is  the first place at which } K_i \neq H_i\\
0 & & \text{if }  t=s;\; \;\;H_i=K_i, \,\;\forall 1 \leq i \leq s \\
\end{array} \right.
\]
The  function $d: Y' \times Y' \rightarrow\mathbb{R} \cup \{\infty\}$ is a \emph{semi-metric} if for all $P,P',P'' \in Y'$, $d$  satisfies  $d(P,P')=d(P',P)$ (symmetry) and  $d(P,P'')\, \leq \,d(P,P')\,+d(P',P'')$ 
(triangle inequality). A standard argument shows:

\begin{lem}
The function $d$ is a semi-metric.
\end{lem}
\begin{proof}
Let $P,P',P'' \in Y'$. Clearly,  $d(P,P')=d(P',P)$. Assume  $d(P,P')=2^{-k}$,  $d(P',P'')=2^{-\ell}$, and $d(P,P'')=2^{-m}$. If $k>1$ or $\ell>1$, then   $m=min\{k, \ell\}$  and   $d(P,P'')=2^{-(min\{k,\ell\})} \, \leq \,2^{-k}\,+2^{-\ell}$. If $k=\ell=1$, then $m \geq 1$ and $d(P,P'')=2^{-m}<1$.
\end{proof}
A \emph{metric} is a semi-metric with 
the additional requirement that $d(P,P') = 0$  implies $P=P'$.
Identifying points with zero distance in a semi-metric $d$ is an equivalence relation that leads to a  metric $\hat{d}$.    The function $\hat{d}$ is then a metric in  $Y'\big/\equiv$, with $P\equiv P'$ if and only if $d(P,P') = 0$. If the answer to Question \ref{ques} is positive, then 
 $Y'\big/\equiv$ is the same as the quotient of $Y'$ by the action of $F_n$. We denote $Y'\big/\equiv$  by $Y$ and  $\hat{d}$ by  $\rho$.

We denote by $B_r(P_0)=\{P \in Y \mid \rho(P,P_0)<2^{-r} \}$, the open ball of radius $2^{-r}$ centered at $P_0$. A set $U \subset Y$  is open if and only if for every point $P \in U$,  there exists $\epsilon>0$ such that $B_\epsilon(P)\subset U$. A space $Y$ is \emph{totally disconnected} if every two distinct points of $Y$ are contained in two disjoint open sets covering the space.
A point $P$ in a metric space   $Y$ is 
 an \emph{isolated point} of $Y$ if there exists a real number $\epsilon>0$, such that $B_\epsilon(P)=\{P\}$. If all the points in  $Y$ are isolated, then $Y$ is \emph{discrete}. The space $Y$  is  \emph{(topologically) discrete} if $Y$ is  discrete as a topological space, that is the  metric  may be different from the  discrete metric.
\begin{thm}
The metric space $Y$ is (topologically) discrete.
\end{thm}
\begin{proof}
We show that all the points in $Y$ are isolated. Let 
$P=\{H_i\alpha_i\}_{i=1}^{i=s}$ in $Y$, with $d_s \geq ...\geq d_1>1$. Then for  $\epsilon < 2^{-(s+1)}$, $B_\epsilon(P)=\{P\}$.
\end{proof}
This implies that  $Y$ is Hausdorff, bounded and totally disconnected,
 facts that could be easily proved directly using $\rho$. A metric space $X$ is \emph{uniformly discrete}, if there exists $\epsilon > 0$ such that for any $x,x' \in X$, $x \neq x'$, $\rho(x,x')>\epsilon$. The space $Y$ is not uniformly discrete.  
 \begin{rem}
 Given an arbitrary group $G$, one can define in the same way the space $Y$, the metric $\rho$ and obtain the same topological properties. The action of $G$ on $Y$ can also be defined in the same way, but it is not necessarily faithful anymore.
 \end{rem}

\begin{proof}[Proof of Theorem \ref{theo4}]

Let $P_0=\{H_i\alpha_i\}_{i=1}^{i=s}$,  with  $1<d_{H_1} \leq ...\leq d_{H_s}$. Let $P \in Y$, $P=\{K_i\beta_i\}_{i=1}^{i=t}$, with  $1<d_{K_1} \leq ...\leq d_{K_t}$.\\
$(i)$  If $\rho(P, P_0)< \frac{1}{2}$, then $K_t=H_s$. So, if  there exists a $d_s$-cycle in $T_{H_s}$, the index $d_s$ appears in $P_0$ and in $P$ at least $p$ times, where $p$ is the least prime dividing $d_s$. Note that   this implies necessarily  $\rho(P, P_0) \leq  2^{-p-1}$.\\
$(ii)$, $(iii)$ If $\rho(P, P_0)< 2^{-(r+1)}$, $2 \leq r \leq s-1$, then $K_t=H_s$,  $K_{t-1}=H_{s-1}$,...,  $K_{t-r}=H_{s-r}$. 
If $k$,  the maximal length of a cycle in $\bigcup\limits_{i=1}^{i=s}T_{H_i}$,  satisfies   $k>d_{H_{s-r}}$, then  $k>d_{K_{t-r}}$ also. Furthermore, $k$ occurs in  $\bigcup\limits_{i=s-r}^{i=s}T_{H_i}$ and also in $\bigcup\limits_{j=t-r}^{j=t}T_{K_j}$, since  $\bigcup\limits_{i=s-r}^{i=s}T_{H_i}\,=\,
\bigcup\limits_{j=t-r}^{j=t}T_{K_j}$. 
If $P_0$ satisfies condition $(i)$ or $(ii)$ of Theorem \ref{theo1}, then $P$   satisfies  the same condition and hence has multiplicity. 

 If $P_0$ satisfies condition $(iii)$ or $(iv)$ of Theorem \ref{theo1}, then $P_0$  has multiplicity, with $d_{H_{s-i}}=d_{H_{s-j}}$ for some $i \neq j$, $0 \leq i,j\leq r$. So, $d_{K_{t-i}}=d_{K_{t-j}}$ also, that is $P$ has multiplicity.
\end{proof}

To conclude, we ask the  following natural question: does there exist a metric on the space of coset partitions $Y'$ that induces a non-discrete topology and yet can give rise to a result of the form of Theorem \ref{theo3} ?

 \section{The  Herzog-Sch\"onheim conjecture for any finitely generated group}
 
  Let $G$ be the group with  presentation $\langle X \mid R \rangle$, with $\mid X \mid =n$. So, $G \simeq\, ^{F_n}\big/_{K}$, with the normal  closure of the set $R$  (not necessarily finite) equal to $K$. There exists the canonical epimorphism $\pi: F_n \rightarrow G$. We show that this partition of $G$ induces a partition of $F_n$ with subgroups of the same indices. 
    \begin{lem}
    Let $G \simeq\,^{F_n}\big/_{K}$,  with canonical epimorphism $\pi: F_n \rightarrow G$.  Let $\{K_ig_i\}_{i=1}^{i=s}$ be a coset  partition of $G$ with $K_i<G$ of index $d_i>1$, $g_i \in G$, $1 \leq i \leq s$.  Let $H_i=\pi^{-1}(K_i)$ and $\alpha_i=\pi^{-1}(g_i)$. Then  $\{H_i\alpha_i\}_{i=1}^{i=s}$ is   a  coset  partition of $F_n$,  with $H_i<F_n$ of index $d_i$, $\alpha_i \in F_n$, $1 \leq i \leq s$.
       \end{lem}
   \begin{proof}
   From the third isomorphism theorem, there is a one-to-one correspondence between the subgroups of $G$ and the subgroups of $F_n$ containing $\operatorname{Ker}(\pi)=K$. We set  $H_i=\pi^{-1}(K_i)$, so $H_i$ are subgroups of $F_n$ of index $d_i$ containing $K$. Let $\alpha_i=\pi^{-1}(g_i)$, that is $\alpha_i$ are  words in $F_n$ representing the elements $g_i\in G$, in particular we can choose the words $g_i$.  We show  $\{H_i\alpha_i\}_{i=1}^{i=s}$ is   a  coset  partition of $F_n$. First, assume by contradiction that $\bigcup\limits_{i=1}^{i=s}H_i\alpha_i \subsetneqq F_n$. So, there exists $w \in F_n$ such that $w \notin\bigcup\limits_{i=1}^{i=s}H_i\alpha_i$. The element $\pi(w)=Kw \in G$, so there exists $1 \leq j \leq s$ such that $Kw\in K_jg_j$. Since $K_jg_j=(^{H_j}\big/_{K})\alpha_j$, there exists $h_j \in H_j$, such that $Kw= Kh_j\alpha_j$, that is $w\in Kh_j\alpha_j$. As $K \subset H_j$, 
   $w \in H_j\alpha_j$.  So, a contradiction, that is $\bigcup\limits_{i=1}^{i=s}H_i\alpha_i =F_n$.  Next, we show that the cosets $H_i\alpha_i$, $1 \leq i \leq s$, are disjoint. Assume   by contradiction that there is $w \in H_i\alpha_i \cap H_j\alpha_j$, $i\neq j$.  So, 
   there exist $h_j \in H_j$, and  $h_i \in H_i$, such that  $w= h_j\alpha_j=h_i\alpha_i$, and $Kw= Kh_j\alpha_j= Kh_i\alpha_i$, with $Kh_j\alpha_j \in (^{H_j}\big/_{K})\alpha_j=K_jg_j$ and $Kh_i\alpha_i \in (^{H_i}\big/_{K})\alpha_i=K_ig_i$. So, a contradiction, that is  the cosets $H_i\alpha_i$, $1 \leq i \leq s$, are disjoint.
  \end{proof}

  \begin{proof}[Proof of Theorem \ref{theo5}]
 If one of the conditions from  Theorem \ref{theo0}, \ref{theo1}, or \ref{theo2}
   is satisfied,  then in the induced partition  $\{H_i\alpha_i\}_{i=1}^{i=s}$, with $H_i$ subgroups of $F_n$ of index $d_i$,  there exists an  index $d_j$ that appears  at least  twice, and this is also true for the partition of $G$.
  \end{proof}
Note that whenever the coset partition $\{H_i\alpha_i\}_{i=1}^{i=s}$ has multiplicity, with $d_i=d_j$ say, then  $H_i$ and $H_j$ are isomorphic. This does not hold necessarily for any finitely generated group $G$. Indeed, if the isomorphism $\varphi_{ij}$ between any pair of isomorphic subgroups $H_i$ and $H_j$ satisfies  $\varphi_{ij}(Ker(\pi))=Ker(\pi)$, then the corresponding subgroups $K_i$ and $K_j$ are also isomorphic.

\bigskip\bigskip\noindent
{ Fabienne Chouraqui,}

\smallskip\noindent
University of Haifa at Oranim, Israel.
\smallskip\noindent
E-mail: {\tt fabienne.chouraqui@gmail.com} {\tt fchoura@sci.haifa.ac.il}

\end{document}